\documentclass[11 pt]{amsart}

\usepackage[hidelinks]{hyperref}
\usepackage{amsmath,amsthm,amssymb,amscd,color, xcolor,mathtools,url,enumerate}
\usepackage[capitalize]{cleveref}
\usepackage{bbm}
\usepackage{tikz}

\newtheorem{thm}{Theorem}[section]
\newtheorem{cor}[thm]{Corollary}
\newtheorem{lem}[thm]{Lemma}
\newtheorem{prop}[thm]{Proposition}

\theoremstyle{definition}
\newtheorem{df}[thm]{Definition}
\theoremstyle{remark}
\newtheorem{rem}[thm]{Remark}
\newtheorem{nota}[thm]{Notation}

\numberwithin{equation}{section}

\newcommand{\R}{\mathbb R}

\newcommand{\indic}{1\!\!1}
\newcommand{\dd}{\mathrm{d}}

\DeclareMathOperator{\supp}{supp}

\author{Charles Collot}
\address{CNRS \& CY Cergy Paris Universit\'e,
 Laboratoire Analyse, G\'eom\'etrie et Mod\'elisation (AGM),
2 avenue Adolphe Chauvin, 95300, Pontoise, France}
\email{ccollot@cyu.fr}
\author{Helge Dietert}
\address{Universit\'e Paris Cité and Sorbonne Universit\'e, CNRS,
  Institut de Math\'ematiques de Jussieu-Paris Rive Gauche (IMJ-PRG),
  F-75013, Paris, France}
\email{helge.dietert@imj-prg.fr}
\author{Pierre Germain}
\address{Courant Institute of Mathematical Sciences, New York University,
251 Mercer Street, New York, NY 10003, United States of America.}
\email{pgermain@cims.nyu.edu}


\title[Stability for the Kolmogorov-Zakharov spectrum]{Stability and
  cascades for the Kolmogorov-Zakharov spectrum of wave
  turbulence}

\begin{document}

\begin{abstract}
  We consider the kinetic wave equation arising in wave turbulence to describe
  the Fourier spectrum of solutions to the cubic Schr\"odinger
  equation. This equation has two Kolmogorov-Zakharov steady states
  corresponding to out-of-equilibrium cascades transferring, for the
  first solution mass from \(\infty\) to \(0\) (small spatial scales
  to large scales), and for the second solution energy from \(0\) to
  \(\infty\). After conjecturing the generic development of the two
  cascades, we verify it partially in the isotropic case by proving
  the nonlinear stability of the mass cascade in the stationary
  setting. This constructs non-trivial out-of-equilibrium steady
  states with a direct energy cascade as well as an indirect mass
  cascade.
\end{abstract}

\maketitle

\section{The Kinetic wave equation and turbulence}

\subsection{The equation and its explicit solutions}

The kinetic wave equation for 4-wave interactions is in the isotropic
case
\begin{equation}
\label{KWE}
\partial_t f =  \mathcal{C}(f),
\end{equation}
where the collision operator is given by
\begin{equation}
  \label{collisionoperator}
  \mathcal{C}(f)(\omega_1)
  = \iint_{\omega_2,\omega_3,\omega_4 \geq 0}
  W [(f_1 + f_2)f_3 f_4 - (f_3 + f_4) f_1 f_2] \, \dd\omega_3 \,\dd\omega_4.
\end{equation}
Here, the distribution \(f\) is expressed in terms of the dispersion
relation $\omega (k)=|k|^2$, where $k$ is the momentum. Moreover, we
use the shorthand $f_i = f(t,\omega_i)$ and the notations
\begin{equation*}
  \omega_2 = \omega_3 + \omega_4 - \omega_1,
  \qquad \text{and} \qquad
  W = \frac{\min( \sqrt{\omega_1}, \sqrt{\omega_2}, \sqrt{\omega_3}, \sqrt{\omega_4})}{\sqrt{\omega_1}}.
\end{equation*}

The structure is similar to the Boltzmann equation in kinetic theory,
where the collision operator is quadratic compared to the cubic
interaction in \eqref{collisionoperator}. Like the Boltzmann collision
operator, the evolution preserves formally the mass and energy given by
\begin{equation}
  \label{eq:mass-energy-integral}
  M(f) = \int_0^\infty  f(t,\omega)\, \omega^{1/2} \,\dd\omega
  \qquad \text{and} \qquad
  E(f) = \int_0^\infty  f(t,\omega)\, \omega^{3/2} \,\dd\omega.
\end{equation}
As an analogue of the Boltzmann H-theorem, the
entropy
\begin{equation*}
  S(f) = \int_0^\infty \log(f)\, \omega^{1/2} \,\dd\omega
\end{equation*}
is formally increasing as
\begin{equation*}
  \frac{\dd}{\dd t} S(f) =  \frac 14
  \iint
  \min (\sqrt{\omega_1},\sqrt{\omega_2},\sqrt{\omega_3},\sqrt{\omega_4})
  f_1f_2f_3f_4 \left(
    \frac{1}{f_1}+\frac{1}{f_2}-\frac{1}{f_3}-\frac{1}{f_4}\right)^2
  \,\dd\omega_1\,\dd\omega_3\,\dd\omega_4.
\end{equation*}

\medskip

Turning to formal stationary solutions, the first class is the
Rayleigh-Jeans (RJ) solutions $f(k) = \frac{1}{\alpha + \beta \omega}$
with $\alpha, \beta>0$ where the collision integral vanishes as
$(f_1 + f_2)f_3 f_4 - (f_3 + f_4) f_1 f_2=0$. They correspond to
statistical equilibria giving equipartition of a linear combination of
mass and energy.

More interesting are the Kolmogorov-Zakharov (KZ) solutions which
correspond to out-of-equilibrium dynamics. They are given by $\omega^{-7/6}$ and $\omega^{-3/2}$,
see Subsection \ref{basiccancellation}, and correspond respectively to a constant flux of mass (particles)
from infinite to zero frequency, and a constant flux of energy from
zero to infinite frequency.

Among all these stationary solutions, the KZ spectrum $\omega^{-7/6}$
is a well-defined solution, while $\omega^{-3/2}$ and the RJ solutions
are only formal solutions as the collision integral is not absolutely
converging. Details regarding the well-definition of the collision integral are given in Section \ref{sec:criticalpowers}.

\subsection{Wave turbulence and dual cascade}

The classical theory of two-dimensional hydro\-dynamic
turbulence~\cite{KM,frisch-1995-turbulence} features a dual cascade:
kinetic energy is transported to smaller frequencies, and enstrophy to
higher frequencies. This behavior is different from three-dimensional
hydro\-dynamic turbulence, and is explained by the two conservation
laws associated to the two-dimensional Euler equation: energy and
enstrophy.

Wave turbulence follows a similar pattern: when the microscopic
model has two conserved quantities, it is expected that the associated
wave kinetic equation will exhibit a dual
cascade~\cite{ZLF,nazarenko-2011-wave}. This is typically the case for
problems whose Hamiltonian only contain even terms, for instance the
MMT toy model~\cite{ZDP},
elasticity~\cite{duering-josserand-rica-2017-wave}, gravitational
waves~\cite{GN}, or the nonlinear Schr\"odinger
equation~\cite{zakharov-1972-langmuir,dyachenko-newell-pushkarev-zakharov-1992-optical},
which is the focus of the present article.

Still in the class of problems exhibiting dual cascades, the (gravity)
water-wave equation has played a very important role in the
development of kinetic wave theory, see~\cite{NL} for a review. It is
speculated
in~\cite{komen-hasselmann-hasselmann-1984-existence-fully-developed-wind-sea-spectrum,Zakharov1}
that steady solutions of the associated forced kinetic wave equation
should display tails at low and high frequencies corresponding to the
inverse and direct cascades, respectively.

\subsection{Proposed dual cascade scenario}

Based on the preceding discussion, we expect that steady solutions
of~\eqref{KWE} will connect the solution $\omega^{-7/6}$ corresponding
to an inverse cascade of mass, with the solution $\omega^{-3/2}$
corresponding to a (direct) energy cascade. It is tempting to propose
the following scenario for the forced kinetic wave equation; we choose
for simplicity in this discussion the force $\phi$ to be smooth and
compactly supported on $(0,\infty)$.

\medskip
\noindent \textit{Static ideal scenario:} The equation $-\mathcal{C}(f)
= \phi$ admits stationary solutions $f$ with tails $\omega^{-7/6}$ and
$\omega^{-3/2}$, as $\omega \to 0$ and $\infty$ respectively. These
tails correspond to outgoing fluxes of mass and energy which
equilibrate the input of these quantities through $\phi$.

\medskip
\noindent \textit{Dynamic ideal scenario:} Solutions of the dynamical
problem $\partial_t f = \mathcal{C}(f) + \phi$ with arbitrary data
converge to one such stationary solution.
\medskip

In the present article, we partially validate the static scenario by
proving the existence of stationary solutions in a neighborhood of
$\omega^{-7/6}$ (mass cascade).

Establishing rigorously the validity of the full ideal scenario seems
very challenging. Compared to common studies in kinetic theory, we
work in \emph{out-of-equilibrium} dynamics so that variational
formulations seem not to apply. Furthermore, the equation is nonlocal,
and the structure of the collision operator is, as we shall see, more
intricate than may seem.

As a side remark, let us mention the Smoluchowski coagulation equation
\cite{wattis-2006,escobedo-mischler-2006-dust-smoluchowski}, which
describes the evolution of dust particles that can stick
together. There the appearance of a flux towards infinity is
well-understood and known as \emph{gelation}. However, the
corresponding collision integral is just quadratic and the evolution
only supports one cascade.

\subsection{Main result}

For any mass flux \(j_M\), we define the scaled Kolmogorov-Zakharov spectrum
$$
\mathfrak f_{j_M}= \left(\frac{j_M}{j_M^*}\right)^{1/3}\frac{1}{ \omega^{7/6}}
$$
where $j_M^*>0$ is a positive constant given by \eqref{id:KZflux}. It
is a steady state solution of \eqref{KWE} as
\(C( \mathfrak f_{j_M})(\omega)=0\) for all \(\omega >0\); a short
proof being given in \cref{lem:KZisstationary}.  The time variations
of the mass and energy densities present an inverse cascade of mass
with a source at $\infty$ and a sink at $0$ of $j_M$ mass per unit of
time, and no cascade of energy:
\begin{align}
  \label{id:massKZ}&- \omega^{1/2} \mathcal C( \mathfrak f_{j_M})=j_M \delta_\infty-j_M \delta_0 \\
  \nonumber & - \omega^{3/2} \mathcal C( \mathfrak f_{j_M})=0.
\end{align}
see \cref{def:mathcalCweak} and \eqref{def:diractdeltas} for the
precise meaning of these equalities.

Our main result shows that this solution is stable if one perturbs the
right-hand side of \eqref{id:massKZ} by adding a source term of finite
mass, and that this also triggers a direct cascade of energy towards
$\infty$.

We formulate the result in terms of the weighted \(L^\infty\) norm
\begin{equation}
\label{defEalphabeta}
  \| f \|_{\alpha,\beta}
  = \sup_{0<\omega<1} \omega^{\alpha} |f(\omega)|
  + \sup_{\omega>1}  \omega^{\beta} |f(\omega)|
\end{equation}
and denote \(E_{\alpha,\beta}(\R_+)\) the corresponding Banach space.

\begin{thm} \label{thm:main}

For all $0<\delta< 1/12$ there exists $\epsilon>0$  such that the
  following holds true. For all $j_M^\infty>0$ and
  $\phi \in L^\infty_{loc}(\mathbb{R}_+)$ satisfying
  \begin{equation} \label{th:id:massinput}
    \| \phi \|_{3/2-\delta,3/2+\delta} \leq \epsilon j_M^\infty,
  \end{equation}
  there exists a solution $f=  \mathfrak f_{j_M^\infty}+g$ to
  \begin{equation*}
   - \mathcal{C}(f)(\omega) =  \phi(\omega), \qquad \forall \omega>0
      \end{equation*}
      with
  $$
  f \geq 0 \qquad \mbox{and} \qquad
  \| g \|_{7/6,7/6+\delta} \lesssim \epsilon j_M^\infty.
  $$

  Furthermore, it satisfies:
  \begin{enumerate}[(i)]
  \item \emph{Stability estimates}. for all $\omega>0$,
    \begin{equation}     \label{th:bd:stability}
      f(\omega)=
      \begin{cases}
        \left(\frac{j_M^0}{j_M^*}\right)^{1/3} \omega^{-7/6}\left(1+O(\epsilon \omega^{\delta})\right)
        & \mbox{for } \omega\leq 1,\\
         \left(\frac{j_M^\infty}{j_M^*}\right)^{1/3} \omega^{-7/6} \left(1+O(\epsilon \omega^{-\delta})\right)
        &\mbox{for }\omega> 1,
      \end{cases}
    \end{equation}
    with the mass balance

    \begin{equation}       \label{th:id:massbalance}
      j_M^0= j_M^\infty + \int_0^\infty \omega^{1/2}\phi\, \dd\omega.
    \end{equation}
  \item \emph{Mass and energy cascades.} In the sense of \cref{def:mathcalCweak},
   \begin{align}
      \label{th:id:massflux}
    -  \omega^{1/2} \mathcal C(f)
      &= \omega^{1/2}\phi+j_M^\infty \delta_\infty-j_M^0 \delta_0 \\
      \label{th:id:energyflux}
    -  \omega^{3/2}\mathcal C(f)
      &= \omega^{3/2}\phi- \left(\int_0^\infty \tilde \omega^{3/2}\phi\, \dd\tilde \omega\right)\delta_\infty,
    \end{align}
where the second identity holds
    provided
    $\int \tilde \omega^{3/2}\phi=\lim_{R\to \infty}\int_0^R \tilde
    \omega^{3/2}\phi$ exists. Its mass and energy fluxes, defined by
    \eqref{def:JM} and \eqref{def:JE}, satisfy:
    \begin{equation} \label{th:id:fluxes}
      J_M(f)(\omega)= - j_M^0+\int_0^\omega \tilde \omega^{1/2}\phi(\tilde
      \omega)\,\dd\tilde \omega
      \quad \text{and} \quad
      J_E(f)(\omega)=\int_0^\omega \tilde \omega^{3/2}\phi(\tilde
      \omega) \,\dd\tilde \omega .
    \end{equation}
  \end{enumerate}
\end{thm}

\begin{rem}[Positivity]
The solutions we construct are non-negative. As for the forcing $\phi$, it is required to be small, but not to have a sign, even though the physically relevant case is $\phi \geq 0$.\end{rem}

\begin{rem}[Cascades]
  The identity \eqref{th:id:massflux} shows that there is a source of
  $j_M^\infty$ mass of particles at $\infty$ per unit of time, which,
  together with the additional mass added by the forcing $\phi$, are
  dissipated in a sink of
  $j_M^0=j_M^\infty+\int \sqrt{\omega}\phi\, \dd\omega$ mass at $0$. The
  identity \eqref{th:id:energyflux} shows that there is no source of
  energy at $0$ and a sink of energy of
  $\int \omega^{3/2}\phi\, \dd\omega$ at $\infty$. Therefore, this
  solution presents an indirect cascade of mass as well as a direct
  cascade of energy.
\end{rem}

\begin{rem}[Optimality]
The condition \eqref{th:id:massinput} for some $\delta>0$ is optimal
  since for $\delta=0$ the forcing $\phi$ would have infinite mass. The upper bound $\delta<1/12$ allows for the shortest proof, avoiding to track logarithmic losses appearing in it; we did not try to optimise the range of $\delta$.
\end{rem}

\begin{rem}[Uniqueness]
  Finally, our proof shows that the constructed solution is unique in a neighbourhood of $\mathfrak f_{j_M^\infty}$.
\end{rem}

\subsection{Mathematical literature}

The kinetic wave equation~\eqref{KWE} is closely related to the
Uhling-Uhlenbeck equation (also called Boltzmann-Bose-Einstein
equation for hard spheres, or bosonic Nordheim equation). Both
equations share the same cubic terms, but differ in lower order
(quadratic) terms; their dynamics have many common features.

The basic theory of the Uhling-Uhlenbeck equation was established
in~\cite{lu-2004-boltzmann-bose-einstein,lu-2005-boltzmann-bose-einstein,lu-2013-boltzmann-bose-einstein,escobedo-velazquez-2015-schroedinger}:
existence of weak solutions, development of singularities
(Bose-Einstein condensation), and large-time
behavior. Escobedo-Mischler-Velazquez~\cite{escobedo-mischler-velazquez-2007-uehling-uhlenbeck,escobedo-mischler-velazquez-2008-singular-uehling-uhlenbeck,escobedo-2012-non}
further studied singular solutions corresponding to the KZ
solution $\mathfrak f_M$. They investigated linear stability
(following~\cite{balk-zakharov-1998-stability-kolmogorov}), and local
well-posedness in their neighborhood; see~\cite{EVa} for RJ
solutions. These articles will provide the foundation of the present
work.

Finally,~\cite{escobedo-velazquez-2015-schroedinger} focused
specifically on~\eqref{KWE}, proving the existence of weak solutions,
Bose-Einstein condensation, and characterizing large-time
dynamics. While all of the above-mentioned works dealt with the
spherically symmetric case,~\cite{germain-ionescu-tran-2020-optimal}
gives optimal local well-posedness result for the kinetic wave
equation without symmetry assumptions.

The kinetic equation for three-wave interactions involve a quadratic
integral kernel. Related phenomena have been studied, see
\cite{soffer-tran-2020,RST} and references therein.

The derivation of the kinetic wave equation from weakly turbulent
solutions to the cubic nonlinear Schr\"odinger equation in suitable
regimes has received a large attention recently, see
\cite{buckmaster-germain-hani-shatah-2021-onset-schroedinger,CG1,deng-hani-2021-nls,CG2,ACG}
for successive attempts and
\cite{lukkarinen-spohn-2011-weakly-schroedinger} for an earlier result
at equilibrium. Its full validity for certain regimes of parameters
has then been obtained in \cite{DH2,DH3}. Additional related recent
results on the kinetic description of wave turbulence may be
found in \cite{F,dymov-kuksin-2020-zakharov-lvov,dS,ST,M}.

\section{Notations}

We will simply write
\begin{equation*}
\mathfrak{f}(\omega)=\mathfrak f_{j_M^*}(\omega)= \frac{1}{ \omega^{7/6}}.
\end{equation*}
For the collision operator we introduce the trilinear operator
\begin{equation}
  \label{eq:collision-3-4-trilinear}
  \bar{\mathcal{C}}(f,g,h) = 2 \iint_{0<\omega_3<\omega_4,0<\omega_2}
  W [(f_1+f_2) g_3 h_4 - (g_3+g_4) h_1 f_2]\, \dd \omega_3\, \dd \omega_4
\end{equation}
and the symmetric form
\begin{equation} \label{id:polarized}
  \mathcal{C}(f,g,h) = \frac 16
  \left( \bar{\mathcal{C}}(f,g,h) + \bar{\mathcal{C}}(f,h,g)
    + \bar{\mathcal{C}}(g,f,h) + \bar{\mathcal{C}}(g,h,f)
    + \bar{\mathcal{C}}(h,f,g) + \bar{\mathcal{C}}(h,g,f)  \right),
\end{equation}
By construction, $\mathcal C(f)=\mathcal C(f,f,f)$.

Note that all forthcoming integrals will be implicitly performed on
the domain of integration
$\{ \omega_1,\omega_2,\omega_3,\omega_4>0, \
\omega_1+\omega_2=\omega_3+\omega_4, \ \omega_3<\omega_4 \}$, which we stop mentioning
explicitly. It is possible to assume that $\omega_3<\omega_4$ due to the symmetry between these variables.

For the estimates we decompose the collision integral according to the
minimal \((\omega_i)_{i=1,\dots,4}\) as
\begin{equation*}
  \mathcal{C}(f)(\omega_1)
  = \iint_{\omega_2,\omega_3,\omega_4 \geq 0, \ \omega_3<\omega_4}
  \left[ \indic_{\omega_1 \textup{min}} + \indic_{\omega_2
      \textup{min}} + \indic_{\omega_3 \textup{min}}   \right]
  \dots \,\dd\omega_3 \,\dd\omega_4
  = \mathcal{C}^1 + \mathcal{C}^2 + \mathcal{C}^3
\end{equation*}
where
\begin{equation*}
 \indic_{\omega_i \textup{min}} =\indic \{ \omega_i<\omega_j \ \mbox{ for } \ j\neq i\}
\end{equation*}
and we extend naturally this notation to other operators, e.g.
$\bar C^i$.

According to \(E_{\alpha,\beta}\) defined in~\eqref{defEalphabeta}, we define the weight
\begin{equation*}
  \rho(\omega)
  = \rho^{\alpha,\beta}(\omega)
  = \indic_{\omega < 1} \omega^{-\alpha} +  \indic_{\omega > 1} \omega^{-\beta}.
\end{equation*}
We obtain the space $E_{\alpha,\beta}(\mathbb{R})$ through the change
of coordinates
\begin{equation}
  \label{eq:change-omega-ex}
  \omega=e^x, \qquad f(\omega) = F(x);
\end{equation}
in other words the norm of $E_{\alpha,\beta}(\mathbb{R})$ is
\begin{equation*}
  \| F(x) \|_{\alpha,\beta}
  = \sup_{x < 0} e^{\alpha x} |F(x)| + \sup_{x>0} e^{\beta x} |F(x)|
  = \| f(\omega) \|_{\alpha,\beta}.
\end{equation*}
The spaces $E_{\alpha,\beta}(\mathbb{R}_+)$ and
$E_{\alpha,\beta}(\mathbb{R})$ can obviously be identified, and we
will systematically abuse notations by simply denoting
$E_{\alpha,\beta}$.

\section{Critical powers} \label{sec:criticalpowers}

\subsection{List of critical powers}

The equation $\partial_t f = \mathcal{C}(f)$ has a more complicated
structure than the simplicity of the formula for the collision
operator might suggest. This becomes clear once one looks at the
various powers which act as thresholds; we first give a list of these
threshold powers from the least to the most singular at $0$

\medskip

\noindent \underline{Power ${-1}$} As we will see below, the range for
local well-posedness is $\alpha < 1 < \beta$; equivalently, this is
the range for which $\mathcal{C}$ is bounded on
$E_{\alpha,\beta}$. The exponent $\alpha = \beta =1$ corresponds to
the scaling of the equation, and it is not clear whether local
well-posedness holds for $E_{1,1}$. However, the Rayleigh-Jeans
solutions are stationary solutions with decay $\omega^{-1}$ as
$\omega \to \infty$. That the RJ solutions actually solve the equation
is related to cancellations in the integrand of $\mathcal{C}$. For this reason, it seems unlikely
that the equation is locally well-posed in $E_{1,1}$, but it is
reasonable to expect nonlinear stability (or co-dimensional stability) of the RJ solutions in a
stronger topology.

\medskip

\noindent \underline{Power ${-7/6}$} The range $\alpha < 7/6 < \beta$
corresponds to solutions which conserve mass. Furthermore, the first
KZ solution is $\omega^{-7/6}$.

\medskip

\noindent \underline{Power $-5/4$} The integral defining the collision
operator converges absolutely for $\alpha < 5/4$ and $\beta>1$, see
\cref{proplocal} below. This is termed ``locality'' in the
work of Zakharov and collaborators.

\medskip

\noindent \underline{Power $-3/2$} The power $3/2$ gives first of all
the threshold for the finiteness of mass: solutions for which
$\alpha < 3/2 < \beta$ have finite mass. Second, though $\mathcal C$ may not consist of absolutely convergent integrals, it remains well-defined in a weak sense for $\alpha<3/2$ and $\beta>1$, see Lemma~\ref{lemweak}. Finally, the power $3/2$ corresponds to the second KZ solution, namely $\omega^{-3/2}$. Since $3/2 > 5/4$, the locality exponent, this KZ solution is deemed physically irrelevant by Balk and
Zakharov~\cite{balk-zakharov-1998-stability-kolmogorov}.

\subsection{Scaling}
The scaling invariance of the collision operator
$$
\mathcal{C}[f(\lambda \cdot)] = \lambda^{-2} \mathcal{C}[f](\lambda \cdot)
$$
implies that the transformation
$$
f(t,\omega) \mapsto \mu^{1/2} \lambda f(\mu t,\lambda \omega)
$$
leaves the set of solutions invariant.

The scaling (in $\lambda$ only) obtained for $\mu =1$, namely
$f(\omega) \mapsto \lambda f(\lambda \omega)$ leaves (amongst others)
the norm $\| \omega f(\omega) \|_\infty$ invariant. One can expect
local well-posedness to hold in spaces with this scaling (or slightly
more regular). This was indeed confirmed
in~\cite{germain-ionescu-tran-2020-optimal} in the non-radial case.

For the record, we include here the simple proof in the radial case.

\begin{prop} \label{propLWP}

  Assume $0\leq \alpha<1<\beta$, then the operator $\mathcal{C}$ is
  bounded from $E_{\alpha,\beta}$ to itself. As a consequence, the
  equation $\partial_t f = \mathcal{C}(f)$ is locally well-posed in
  $E_{\alpha,\beta}$.
\end{prop}
\begin{proof}
  The estimate for \(\mathcal{C}\) follows directly from \cref{propC}. The local well-posedness in $E_{\alpha,\beta}$ then follows by standard Cauchy-Lipschitz theory.
\end{proof}

\subsection{Local property}

As defined in Balk and
Zakharov~\cite{balk-zakharov-1998-stability-kolmogorov}, a spectrum is
called local if each of the summands defining $\mathcal{C}$ is absolutely
convergent. These authors suggest that physically relevant spectra
must be local.

\begin{prop} \label{proplocal}

  A function~$f \in E_{\alpha,\beta}$ is local if $\beta > 1$ and
  $0\leq \alpha < \frac{5}{4}$.
\end{prop}

This is a direct corollary of
\cref{thm:collision-trilinear-estimate}. When either $\alpha\geq 5/4$
or $\beta\leq 1$, explicit examples can be found following the proof
of \cref{thm:collision-trilinear-estimate}, for which the collision
integrals may no longer be absolutely convergent, hence the optimality
of the range of parameters in \cref{proplocal}.

\subsection{Weak formulation}
For smooth and localized $f$ and $\phi$, symmetrizing leads to the identity
\begin{equation} \label{id:weakdefinitionmathcalC}
\int \sqrt{\omega_1}\mathcal{C}(f) \, \varphi_1 \,d\omega_1
  = \frac 12 \iiint \sqrt{\omega_1} W f_1 f_2 (f_3+f_4)
  \left(  \varphi_3 + \varphi_4- \varphi_1 - \varphi_2 \right)
  \,\dd\omega_1\, \dd\omega_3 \,\dd\omega_4.
\end{equation}
This weak formulation features a cancellation in the term $\left(  \varphi_3 + \varphi_4- \varphi_1 - \varphi_2 \right)$, which allows to make sense of this identity for a larger set of
functions than allowed by Proposition \ref{proplocal}; this is the content of the following lemma.

\begin{lem} \label{lemweak}
  Assume $f\in E_{\alpha,\beta}$ with $\alpha<3/2$ and $\beta>1$ and
  $\varphi \in \mathcal C^\infty_0((0,\infty))$. Then the integral in
  the right-hand side of \eqref{id:weakdefinitionmathcalC} converges
  absolutely.
\end{lem}

One can check that this result is optimal; for either $\alpha\geq 3/2$
or $\beta\geq 1$ there exists $f \in E_{\alpha,\beta}$ for which the
integral is not absolutely convergent if
\(\varphi \in C_0^\infty((0,\infty))\) is nonzero.

\begin{proof}
  By symmetry, we can restrict to the case \(\omega_1<\omega_2\) and \(\omega_3<\omega_4\) and consider
  \begin{equation*}
    I = \iiint_{\omega_1<\omega_2, \omega_3<\omega_4}
    \min\left(\sqrt{\omega_1},\sqrt{\omega_3}\right)
    f_1 f_2 (f_3+f_4)
    \left(\varphi_3 + \varphi_4- \varphi_1 - \varphi_2 \right)
    \,\dd\omega_1\, \dd\omega_3 \,\dd\omega_4.
  \end{equation*}

  Since \(\varphi\) is a compactly supported test function, it can be assumed to have  \(\supp \varphi \subset [2\delta,\delta^{-1}]\) for some
  \(\delta > 0\). Then \(\omega_2 \ge \delta\) if \( \varphi_3 + \varphi_4- \varphi_1 - \varphi_2 \neq 0\). We decompose accordingly
  \begin{align*}
    I & = \underbrace{\int_{\omega_2=\delta}^{\infty}
        \int_{\omega_1=0}^\delta
        \int_{\omega_3=0}^\delta
        \min\left(\sqrt{\omega_1},\sqrt{\omega_3}\right)
        f_1 f_2 (f_3+f_4)\,
        (\varphi_4-\varphi_2)\, \dd \omega_3\, \dd \omega_1\, \dd \omega_2.}_{=I_1}\\
      &+ \underbrace{\int_{\omega_2=\delta}^{\infty}
        \iint_{\max(\omega_1,\omega_3)\geq \delta }
        \min\left(\sqrt{\omega_1},\sqrt{\omega_3}\right)
        f_1 f_2 (f_3+f_4)\,
        (\varphi_4-\varphi_2+\varphi_3-\varphi_1)\, \dd \omega_3\, \dd \omega_1\, \dd \omega_2.}_{=I_2}
  \end{align*}

  The integral $I_1$ contains the contribution of the mass for small
  \(\omega\). As \(\varphi\) is Lipschitz we find that
  \(|\varphi_4-\varphi_2| \le L |\omega_4-\omega_2| = L
  |\omega_3-\omega_1|\). Using also $|f_3+f_4|\lesssim \rho_3$ and $\int_\delta^\infty |f_2|d\omega_2<\infty$ as $\beta>1$, the integral $I_1$ is
  absolutely converging as soon as
  \begin{equation*}
    \int_{\omega_1=0}^\delta
    \int_{\omega_3=0}^\delta
    \min\left(\sqrt{\omega_1},\sqrt{\omega_3}\right)
    \rho_1 \rho_3 |\omega_1-\omega_3|\,
    \dd \omega_3\, \dd \omega_1
  \end{equation*}
  is integrable. It can be bounded by
  \begin{align*}
    \int_{\omega_1=0}^\delta
    \int_{\omega_3=0}^\delta
    \min\left(\sqrt{\omega_1},\sqrt{\omega_3}\right)
    \rho_1 \rho_3 |\omega_1-\omega_3|\,
    \dd \omega_3\, \dd \omega_1
    \le 4
    \int_{\omega_1=0}^\delta
    \int_{\omega_3=\omega_1}^\delta
    \sqrt{\omega_1}
    \omega_1^{-\alpha}
    \omega_3^{-\alpha}
    \omega_3
    \, \dd \omega_3\, \dd \omega_1,
  \end{align*}
 which is finite as soon as \(\alpha < 3/2\). Alternatively, the
  integral is bounded if \(M(\indic (\omega<1)f) < \infty\).

  As for the second integral $I_2$, we bound by brute force
  $ \min\left(\sqrt{\omega_1},\sqrt{\omega_3}\right)\lesssim
  \min(1,\sqrt{\omega_1},\sqrt{\omega_3})$ on the support of
  $ (\varphi_4-\varphi_2+\varphi_3-\varphi_1)$ and
  $|\varphi_4-\varphi_1+\varphi_3-\varphi_1|\lesssim 1$ so that
  \begin{equation*}
    |I_2| \lesssim \int_{\omega_2=\delta}^\infty
    \iint_{\max(\omega_1,\omega_3)\geq \delta }
    \min(1,\sqrt{\omega_1},\sqrt{\omega_3}) \rho_1\rho_2\rho_3\,
    \dd\omega_1\,\dd\omega_2\,\dd\omega_3
    <\infty. \qedhere
  \end{equation*}
\end{proof}

\subsection{How to interpret the equation}
Consider a function
$$
f \in \mathcal{C}^1([0,T],E_{\alpha,\beta}).
$$
When is it a solution of the kinetic wave equation~\eqref{KWE}? In light of the discussion above, we see that the equation can be
given different meanings, depending on the range of $\alpha$ and
$\beta$. We discuss the evolution problem, but the extension to
stationary solutions is obvious.

\medskip

\noindent \underline{If $\alpha < 1 < \beta$}, by Proposition~\ref{propLWP}, the
equation~\eqref{KWE} can be interpreted as an equality in
$\mathcal{C}^1([0,T],E_{\alpha,\beta})$.

\medskip

\noindent \underline{If $\alpha < \frac{5}{4}$, $\beta>1$}, by
\cref{proplocal}, the equation~\eqref{KWE} can be
understood pointwise in $[0,T] \times (0,\infty)$.

\medskip

\noindent \underline{If $\alpha < \frac{3}{2}$ or finite mass and
  $\beta>1$}, by Lemma~\ref{lemweak}, the equation~\eqref{KWE}
can be understood weakly (tested against compactly supported Lipschitz
functions).

\begin{rem}
  This has not been mentioned yet, but the equation~\eqref{KWE} is
  always considered in the present paper in the
  \emph{non-interacting condensate} regime, following the
  terminology of~\cite{escobedo-velazquez-2015-schroedinger}. This
  means that the mass which is transfered to zero momentum is not
  allowed to interact with the rest of the solution, or equivalently
  that it is instantaneously absorbed by an infinite sink.
\end{rem}

\section{Cascades, fluxes of mass and energy}

The fluxes of the equation are a common tool for studying the equation
\cite{dyachenko-newell-pushkarev-zakharov-1992-optical,spohn-2010-kinetics-bose-einstein,escobedo-velazquez-2015-schroedinger}
and they provide a natural notion of solution.

\subsection{Definition of the fluxes}

By \cref{proplocal} and the computation of \cref{basiccancellation},
we have that $\mathcal C(\mathfrak f)(\omega)$ is well-defined
(i.e. by a convergent integral), and equal to $0$ for every
$\omega$. However, ``$\mathcal C(\mathfrak f)(\omega)=0$'' is not an
appealing identity, since one expects $\mathfrak f$ to be a steady
state with constant mass flux from $\infty$ to $0$. Rather, we propose here in \cref{def:mathcalCweak} a weak definition of
$\omega^{1/2}\mathcal C(f)$ and $\omega^{3/2}\mathcal C(f)$ (which
would be the time variation of the mass and energy densities for a
solution of \eqref{KWE}) for general $f\in E_{7/6,7/6}$. The KZ
solution $\mathfrak f$ will then present a source of mass at
$\infty$ and a sink at $0$. These definitions of
$\omega^{1/2}\mathcal C(f)$ and $\omega^{3/2}\mathcal C(f)$ will
moreover be proved to be natural, since corresponding to the
approximation by finite size systems, i.e. to the limit of
$\omega^{1/2}\mathcal C(f_n)$ and $\omega^{3/2}\mathcal C(f_n)$ for
any sequence $f_n\to f$ of smooth and localised functions $f_n$.

We first derive the mass flux. After symmetrizing, for $f$ smooth and
localised we obtain that
\begin{equation}
  \label{symmetrizing}
  \begin{split}
    \int \sqrt{\omega_1}\mathcal{C}(f) \, \varphi_1 \,d\omega_1
    & =  \iint_{\omega_2,\omega_3,\omega_4 \geq 0} \sqrt{\omega_1}W
    [(f_1 + f_2)f_3 f_4 - (f_3 + f_4) f_1 f_2] \varphi_1 \,\dd\omega_1\, \dd\omega_3 \,\dd\omega_4 \\
    & = \int \sqrt{\omega_1} W f_1 f_2 f_3
    \left(  \varphi_3 + \varphi_4- \varphi_1 - \varphi_2 \right) \,\dd\omega_1\, \dd\omega_3 \,\dd\omega_4.
  \end{split}
\end{equation}

We now write
\begin{align*}
  & \varphi_3  - \varphi_2
    = \indic_{\omega_3 > \omega_2} \int_{\omega_2}^{\omega_3} \varphi'
    - \indic_{\omega_2 > \omega_3} \int_{\omega_3}^{\omega_2} \varphi' \\
  & \varphi_4 - \varphi_1
    = \indic_{\omega_4 > \omega_1} \int_{\omega_1}^{\omega_4} \varphi'
    - \indic_{\omega_1 > \omega_4} \int_{\omega_4}^{\omega_1} \varphi' .
\end{align*}

Using the formulas above to express
$\left( \varphi_3 + \varphi_4 - \varphi_1 - \varphi_2 \right) $,
inserting the result and rearranging, we find that
\begin{equation} \label{id:massfluxduality}
  \int \sqrt{\omega}\mathcal{C}(f) \, \varphi \,\dd\omega
  = \int J_M(f) \, \varphi'  \,\dd\omega.
\end{equation}
Above, $J_M$ is a flux whose orientation is from $0$ to $\infty$
with\footnote{We suspect that there is a typo in
  \cite{escobedo-velazquez-2015-schroedinger} for \(J_1\).}
\begin{align}
  \label{def:JM}
  &J_M(f) = J_1(f) +J_2(f) - J_3(f) - J_4(f)\\
  \label{def:J1}
  &J_1(f)(\omega) = \int_{\substack{\omega_1 >0 \\ 0<\omega_2< \omega \\ \omega < \omega_3 < \omega_1 + \omega_2}}  \sqrt{\omega_1} W f_1 f_2 f_3 \,\dd\omega_1\, \dd\omega_2 \,\dd\omega_3 \\
  \label{def:J2}
  &J_2(f)(\omega) = \int_{\substack{0<\omega_1 <\omega \\ \omega_2> \omega-\omega_1>0 \\ 0 < \omega_3 < \omega_1 + \omega_2-\omega}}  \sqrt{\omega_1} W f_1 f_2 f_3 \,\dd\omega_1\, \dd\omega_2 \,\dd\omega_3 \\
  \label{def:J3}
  &J_3(f)(\omega) =  \int_{\substack{\omega_1 >0 \\ \omega_2> \omega \\ 0 < \omega_3 < \omega}}  \sqrt{\omega_1} W f_1 f_2 f_3 \,\dd\omega_1\, \dd\omega_2 \,\dd\omega_3 \\
  \label{def:J4}
  &J_4(f)(\omega) =  \int_{\substack{\omega_1 >\omega \\ 0<\omega_2 \\ \omega_1 +\omega_2 -\omega < \omega_3 < \omega_1+\omega_2}}  \sqrt{\omega_1} W f_1 f_2 f_3 \,\dd\omega_1\, \dd\omega_2 \,\dd\omega_3
\end{align}

Next, integrating by parts in \eqref{id:massfluxduality}, and using
$\int \sqrt{\omega}C(f)=0$ we get
\begin{equation} \label{id:energyfluxduality}
\int \omega^{3/2}\mathcal{C}(f) \, \varphi \,d\omega = \int J_E(f) \, \varphi ' \,\dd\omega,
\end{equation}
where
\begin{equation} \label{def:JE}
  J_E(f)(\omega)
  = \omega J_M(\omega)-\int_0^\omega J_M(\tilde \omega)\, \dd\tilde \omega.
\end{equation}

\subsection{Locality of the fluxes}

\begin{lem} \label{lem:localflux}

  If $\alpha<\frac 54$ and $\beta>1$ then the fluxes
  $(J_i)_{1\leq i \leq 4}$ are local (i.e. the integrals converge
  absolutely).
\end{lem}

\begin{proof}
  By a rescaling of variables, it suffices to take $\omega=1$ and
  prove $J_n(\rho)(1)<\infty$ for $1\leq n\leq 4$. We introduce
  $\tilde W_i=\sqrt{\omega_i}$ and
  $\tilde W_{i,j}=\min (\sqrt{\omega_i},\sqrt{\omega_j})$.

  Observe first the following integrals are finite since $\alpha<5/4$
  and $\beta>1$:
  \begin{equation} \label{bd:localflux1}
    \int_{\omega_i>1} \rho_i\,\dd\omega_i<\infty,
    \qquad \int_{\omega_i<1}\tilde W_i\rho_i\,\dd\omega_i<\infty,
    \qquad \int_{\omega_i,\omega_j<1}\tilde W_{i,j}\rho_i\rho_j\,
    \dd\omega_i \, \dd \omega_j<\infty.
  \end{equation}

  For the first flux:
  \begin{align*}
    J_1(\rho)(1)
    & =\int_{\substack{\omega_1,\omega_2<1\\
    1<\omega_3<\omega_1+\omega_2}}
    \sqrt{\omega_1}W\rho_1\rho_2\rho_3\,
    \dd\omega_1\,\dd\omega_2\,\dd\omega_3
    +\int_{\substack{\omega_1>1 \\ \omega_2<1\\ 1<\omega_3<\omega_1+\omega_2}} \sqrt{\omega_1}W\rho_1\rho_2\rho_3 \,\dd\omega_1\,\dd\omega_2\,\dd\omega_3 \\
    & \leq \int_{\substack{\omega_1,\omega_2<1\\ 1<\omega_3}} \tilde
    W_{1,2} \rho_1\rho_2\rho_3\, \dd\omega_1\,\dd\omega_2\,\dd\omega_3
    +\int_{\substack{\omega_1,\omega_3 >1 \\ \omega_2<1}} \tilde W_2\rho_1\rho_2\rho_3 \,\dd\omega_1\,\dd\omega_2\,\dd\omega_3 \quad <\infty
  \end{align*}
  using a combination of the identities in \eqref{bd:localflux1}. For the second, similarly:
  \begin{align*}
    J_2(\rho)(1) & = \int_{\substack{\omega_1<1/2 \\ 1-\omega_1<\omega_2 \\ \omega_3<\min (\omega_1+\omega_2-1,1)}}  ...+ \int_{\substack{\omega_1<1/2 \\ 1-\omega_1<\omega_2 \\ 1<\omega_3 <\omega_1+\omega_2-1}}  ...+\int_{\substack{1/2 <\omega_1<1 \\ 1-\omega_1<\omega_2<1\\ 0<\omega_3<\omega_1+\omega_2-1}} ...\\
                 &\qquad +\int_{\substack{1/2 <\omega_1<1 \\ \omega_2>1\\ 0<\omega_3<1/2}}...+\int_{\substack{1/2 <\omega_1<1 \\ \omega_2>1\\ 1/2<\omega_3<\omega_1+\omega_2-1}}...  \\
    \leq &   \int_{\substack{\omega_1,\omega_3<1 \\ \omega_2>1/2}}  \tilde W_{1,3} \rho_1\rho_2\rho_3+ \int_{\substack{\omega_1<1/2 \\ 1/2<\omega_2 \\ 1<\omega_3 }}\tilde W_1 \rho_1\rho_2\rho_3 +\int_{\substack{1/2 <\omega_1<1 \\ 0<\omega_2,\omega_3<1}}\tilde W_{2,3} \rho_1\rho_2\rho_3\\
                 &\qquad  +\int_{\substack{1/2 <\omega_1<1 \\ \omega_2>1\\ 0<\omega_3<1}} \tilde W_{3}\rho_2\rho_3+\int_{\substack{1/2 <\omega_1<1 \\ \omega_2>1\\ 1/2<\omega_3}} \rho_2\rho_3 \quad   <\infty.
  \end{align*}
  For the third and fourth, analogously:
  \begin{align*}
    J_3(\rho)(1) &=   \int_{\substack{\omega_1<1 \\ \omega_2>1 \\ \omega_3<1}}  \tilde W_{1,3} \rho_1\rho_2\rho_3+  \int_{\substack{\omega_1>1 \\ \omega_2>1 \\ \omega_3<1}}  \tilde W_{3} \rho_1\rho_2\rho_3 <\infty,\\
    J_4(\rho)(1) &=   \int_{\substack{\omega_1>1 \\ \omega_2>1 \\ \omega_1+\omega_2-1<\omega_3<\omega_1+\omega_2}} ... +  \int_{\substack{\omega_1>1 \\ \omega_2<1 \\ \omega_1+\omega_2-1<\omega_3<1}} ... +  \int_{\substack{\omega_1>1 \\ \omega_2<1 \\  1<\omega_3< \omega_1+\omega_2}} ... \\
                 &\leq \int_{\substack{\omega_1>1 \\ \omega_2>1 }} \rho_1\rho_2  +  \int_{\substack{\omega_1>1 \\ \omega_2<1 \\ \omega_3<1}} \tilde W_{2,3}\rho_1\rho_2\rho_3+  \int_{\substack{\omega_1>1 \\ \omega_2<1 \\  1<\omega_3}} \tilde W_2 \rho_1\rho_2\rho_3 \quad <\infty
  \end{align*}
  where for the first sum in $J_4$ we used that
  $\rho_3\sqrt{\omega_1}W\lesssim 1$ in the domain of
  integration. This concludes the proof.
\end{proof}

By a rescaling of the variables, one easily checks that the flux of
mass of the KZ solution $\mathfrak f=\omega^{-7/6}$
is constant:
\begin{equation} \label{id:KZflux}
  \forall \omega>0, \qquad J_M(\mathfrak f)(\omega)=-j_M^*<0.
\end{equation}
Its negativity, see
\cite{spohn-2010-kinetics-bose-einstein,escobedo-velazquez-2015-schroedinger},
indicates mass flowing from infinity to $0$. Combining \eqref{def:JM}
and \eqref{id:KZflux} shows there is no flux of energy:
\begin{equation} \label{id:KZflux2}
  \forall \omega>0, \qquad J_E(\mathfrak f)(\omega)=0.
\end{equation}

\begin{cor} \label{cor:fluxes}

  Let $f\in E_{7/6,7/6}$. Then
  \begin{enumerate}[(i)]
  \item There holds $J_M(f)\in L^\infty(0,\infty)$ and
    $\omega^{-1}J_E(f)\in L^\infty(0,\infty)$ with
    \begin{equation} \label{bd:JMJE}
      \| J_M(f)\|_{L^\infty(0,\infty)}+\| \omega^{-1}J_E(f)\|_{L^\infty(0,\infty)}\lesssim \| f\|_{7/6,7/6}^3.
    \end{equation}
  \item If $f_n$ is a bounded sequence in $E_{7/6,7/6}$ such that
    $f_n\rightarrow f$ almost everywhere, then $J_M(f_n)\to J_M(f)$
    and $J_E(f_n)\to J_E(f)$ uniformly on every compact sets of
    $(0,\infty)$ .
  \end{enumerate}
\end{cor}

\begin{proof}

  We mentioned right before the corollary that the integrals defining
  $J_k(\mathfrak f_M)(\omega)$ for $k=1,2,3,4$ are convergent, with a
  value that does not depend on $\omega$. This directly implies part
  $(i)$ of the corollary.

  For the second one, fix $0<\omega_-<\omega_+$ and
  $k\in \{1,2,3,4\}$. Decompose any function $g$ as
  $g=g_1^\epsilon+g_2^\epsilon$ where
  $g_1^\epsilon=\indic (\epsilon<\omega<\epsilon^{-1})g$. Then
  $$
  J_k(f)-J_k(f_n)=\sum_{1\leq i_1,i_2,i_3\leq 2} J_k(f^\epsilon_{i_1},f^\epsilon_{i_2},f^\epsilon_{i_3})-J_k(f^\epsilon_{n,i_1},f^\epsilon_{n,i_2},f^\epsilon_{n,i_3})
  $$
  where $J_k(a,b,c)$ is given by the corresponding expression in
  \eqref{def:J1}-\eqref{def:J4}, with $f_1,f_2,f_3$ replaced by
  $a,b,c$ respectively. On the one hand, by \cref{lem:localflux}, if
  $(i_1,i_2,i_3)\neq (1,1,1)$ then
  $J_k(f^\epsilon_{i_1},f^\epsilon_{i_2},f^\epsilon_{i_3})\to 0$ and
  $J_k(f^\epsilon_{n,i_1},f^\epsilon_{n,i_2},f^\epsilon_{n,i_3})\to 0$
  uniformly for $\omega \in [\omega_-,\omega_+]$ and $n\in \mathbb
  N$. On the other hand, for each $\epsilon>0$, still by
  \cref{lem:localflux},
  $J_k(f^\epsilon_{n,1},f^\epsilon_{n,1},f^\epsilon_{n,1})\to
  J_k(f^\epsilon_{1},f^\epsilon_{1},f^\epsilon_{1})$ uniformly on
  $[\omega_-,\omega_+]$ as $n\to \infty$. Combining, we obtain that
  $J_M(f_n)\to J_M(f)$ uniformly on $[\omega_-,\omega_+]$. Using this
  convergence for any $0<\omega_-<\omega_+$, the identity
  \eqref{def:JM} and the first bound in \eqref{bd:JMJE} shows that
  $J_E(f_n)\to J_E(f)$ as well uniformly on any compact.
\end{proof}

\subsection{Cascades}

We are able to make sense of mass and energy cascades by considering
their fluxes in the framework of distributions. We define the space of
test functions $\mathcal D([0,\infty))$ as the set of all functions
$\varphi \in \mathcal C^\infty([0,\infty))$ such that $\varphi'$ has
compact support. This space can be equipped with a similar topology as
the space of usual test functions for usual distributions. Hence, we
abuse notations and still call distributions linear forms on
$\mathcal D$.

\begin{df} \label{def:mathcalCweak}

  For $f\in E_{7/6,7/6}$, we define $\sqrt{\omega}\mathcal C(f)$ and
  $\omega^{3/2} \mathcal C(f)$ in the distributional sense by:
  \begin{align} \label{def:Cweaksense}
    \forall \varphi \in \mathcal D, \qquad
    \langle \sqrt{\omega} \mathcal C(f),\varphi\rangle=\int_{0}^\infty J_M(f)(\omega)\, \varphi'(\omega)\, \dd\omega,\\
    \label{def:Cweaksense2} \forall \varphi \in \mathcal D, \qquad
    \langle \omega^{3/2} \mathcal C(f),\varphi\rangle=\int_{0}^\infty J_E(f)(\omega)\, \varphi'(\omega)\, \dd\omega.
  \end{align}
\end{df}

\begin{rem}
  These weak definitions $\sqrt{\omega} \mathcal C(f)$ and
  $\omega^{3/2} \mathcal C(f)$ for any $f\in E_{7/6,7/6}$ are natural
  because of \cref{cor:fluxes}. Indeed, let us approximate $f$ by a
  sequence $f_n$ of smooth and localised functions. Then for all
  $\varphi\in \mathcal D$ we have, on the one hand by item (ii) that
  $\langle \sqrt{\omega} \mathcal C(f_n),\varphi\rangle \rightarrow
  \langle \sqrt{\omega} \mathcal C(f),\varphi\rangle $, and on the
  other hand that
  $\langle \sqrt{\omega} \mathcal C(f_n),\varphi\rangle=\int \mathcal
  C(f_n)\varphi$ because of \eqref{id:massfluxduality}. Thus the
  formula \eqref{def:Cweaksense} is such that
  $ \langle \sqrt{\omega} \mathcal C(f),\varphi\rangle=\lim_{n\to
    \infty}\int \mathcal C(f_n)\varphi$ is the limit obtained from any
  approximation of $f$ by non-singular functions $f_n$. The same
  approximation holds true for $\omega^{3/2}f$ as well.
\end{rem}

Thanks to \eqref{id:KZflux}-\eqref{id:KZflux2} we have for the KZ
steady state $\mathfrak f$ that, in the sense of
\eqref{def:Cweaksense}-\eqref{def:Cweaksense2}:
$$
-\sqrt{\omega} \mathcal C(\mathfrak f)=j_M^*\delta_\infty-j_M^*\delta_0, \qquad -\omega^{3/2} \mathcal C(\mathfrak f)=0,
$$
where we introduced the Dirac deltas:
\begin{equation} \label{def:diractdeltas}
\forall \varphi \in \mathcal D, \qquad \langle \delta_0, \varphi\rangle =\varphi (0) \qquad \mbox{and} \qquad \langle \delta_\infty, \varphi\rangle =\lim_{\omega\to \infty}\varphi (\omega).
\end{equation}
This shows an indirect cascade of mass from $\infty$ to $0$ and no
cascade of energy.

The next proposition relates the fluxes of a solution to
$-\mathcal C(f)=\phi$ to its asymptotic behaviours near $0$ and
$\infty$.

\begin{prop} \label{pr:frompointwisetofluxes}
  Assume $f\in E_{7/6,7/6}$ and $\phi \in L^1(\sqrt{\omega}d\omega)$ solve
  $$
  \forall \omega>0, \qquad - \mathcal C(f)(\omega)=\phi(\omega),
  $$
  with
  $$
  f(\omega)=c_0 \omega^{-7/6}+o(\omega^{-7/6}) \mbox{ as }\omega \downarrow 0 \quad \mbox{and }\quad f(\omega)=c_\infty \omega^{-7/6}+o(\omega^{-7/6}) \mbox{ as }\omega \rightarrow \infty
  $$
  for some $c_0,c_\infty \in \mathbb R$. Then, in the sense of
  distributions given by \cref{def:mathcalCweak}:
  \begin{align} \label{id:weakformulation}
 &   -  \sqrt{\omega}\mathcal C(f)=c_\infty^3 j_M^* \delta_\infty-c_0^3 j_M^* \delta_0+\omega^{1/2}\phi,\\
  &  \label{id:weakformulationthird} - \omega^{3/2}\mathcal C(f)=\omega^{3/2}\phi-E(\phi) \delta_\infty
  \end{align}
  where the second identity holds provided
  $E(\phi)=\lim_{R\to \infty} \int_0^R \omega^{3/2}\phi \,\dd\omega$
  exists, and there holds:
  \begin{equation} \label{id:compatibilityweakformulation}
    c_0^3=c_\infty^3+\frac{M(\phi)}{j_M^*}.
  \end{equation}
  Moreover,
  \begin{equation} \label{id:weakformulationsecond}
    J_M(f)(\omega)=-c_0^3j_M^*+\int_0^\omega \tilde
    \omega^{1/2}\phi(\tilde \omega)\,\dd\tilde \omega ,
    \qquad J_E(f)(\omega)=\int_0^\omega \tilde \omega^{3/2}\phi(\tilde \omega)\,\dd\tilde \omega .
  \end{equation}
\end{prop}

\begin{proof}
  Let $f_\lambda (\omega)=\lambda^{-7/6}f(\omega/\lambda)$ for
  $\lambda>0$. Then $f_{\lambda}$ is bounded uniformly in
  $E_{7/6,7/6}$ for all $\lambda>0$, with
  $f\rightarrow c_0\omega^{-7/6}$ as $\lambda\to \infty$ and
  $f\rightarrow c_\infty\omega^{-7/6}$ as $\lambda \to 0$, in
  $L^\infty_{loc}((0,\infty))$. By \cref{lem:localflux} and
  \eqref{id:KZflux}, this implies $J(f_\lambda)(1)\to c_0^{3}j_M^*$
  and $J(f_\lambda)(1)\to c_\infty^{3}j_M^*$ respectively. Since by
  rescaling, $J(f)(\lambda)=J(f_{\lambda^{-1}})(1)$ we get
  \begin{equation} \label{id:weakformulationlimitfluxes}
    \lim_{\omega \downarrow 0} J(f)(\omega)= - c_0^{3}j_M^* \quad \mbox{and} \quad \lim_{\omega \rightarrow \infty} J(f)(\omega)= - c_\infty^{3} j_M^*.
  \end{equation}
  Fix now $\varphi \in \mathcal D$. Let $\chi$ be a smooth cut-off
  function with $\chi(\omega)=1$ for $\omega\leq 1$ and
  $\chi (\omega)=0$ for $\omega \geq 2$. For $R>0$ we let
  $\chi^R(\omega)=\chi(\omega/R)$. For $\epsilon>0$ we decompose
  $\varphi=\varphi^\epsilon_1+\varphi_2^\epsilon+\varphi^\epsilon_3$
  where $\varphi_1^\epsilon=\chi^\epsilon \varphi $ and
  $\varphi^\epsilon_3=(1-\chi^{1/\epsilon})\varphi$ so that
  \begin{equation} \label{id:weakformulation1}
    \langle \sqrt{\omega}\mathcal C(f),\varphi \rangle
    =\sum_{k=1}^3\langle \sqrt{\omega}\mathcal C(f),\varphi_k^\epsilon \rangle.
  \end{equation}
  On the one hand, by \eqref{def:Cweaksense} and
  \eqref{id:weakformulationlimitfluxes} we have:
  \begin{equation} \label{id:weakformulation2}
    \langle \sqrt{\omega}\mathcal C(f),\varphi_1^\epsilon \rangle \to c_0^3j_M^*\varphi(0) \qquad \mbox{and} \qquad \langle \sqrt{\omega}\mathcal C(f),\varphi_3^\epsilon \rangle \to - c_\infty^3j_M^* \lim_{\omega \to \infty}\varphi(\omega)
  \end{equation}
  as $\epsilon\to 0$. On the other hand, for each fixed $\epsilon>0$,
  the derivation of the formula \eqref{id:massfluxduality} is valid
  for $\varphi_2^\epsilon$ as it has compact support in $(0,\infty)$
  and using \cref{lem:localflux}, which combined with the identity
  $-\mathcal C(f)(\omega)=\phi(\omega)$ for all $\omega>0$ gives:
  \begin{equation*}
    \langle \sqrt{\omega}\mathcal C(f),\varphi_2^\epsilon \rangle
    = - \int_0^\infty \sqrt{\omega}\phi \varphi_3^\epsilon \,\dd\omega.
  \end{equation*}
  Hence
  $\langle \sqrt{\omega}\mathcal C(f),\varphi_2^\epsilon \rangle
  \rightarrow - \int_0^\infty \sqrt{\omega}\phi \varphi d\omega$ as
  $\epsilon \to 0$, which, combined with \eqref{id:weakformulation1}
  and \eqref{id:weakformulation2} implies \eqref{id:weakformulation}.

  The identity \eqref{id:compatibilityweakformulation} then directly
  follows upon using \eqref{id:weakformulation} with the test function
  $\varphi=1$.

  The identity \eqref{id:weakformulation} implies the first identity
  in \eqref{id:weakformulationsecond}. Injecting
  \eqref{id:compatibilityweakformulation} and the first identity in
  \eqref{id:weakformulationsecond} in the expression \eqref{def:JM}
  for $J_E$ shows the second identity in
  \eqref{id:weakformulationsecond}. This, in turns, injected in
  \eqref{def:Cweaksense2}, implies the identity
  \eqref{id:weakformulationthird} .
\end{proof}

\section{Proof of the main result}

This section is devoted to the proof of \cref{thm:main}. By scaling
invariance, it suffices to prove the result in the case $j_M^0=j^*_M$,
and the general result follows. We thus aim at solving
\begin{equation} \label{id:mainequation}
  \forall \omega>0, \qquad \qquad - \mathcal{C}(f)(\omega) =  \phi(\omega),
\end{equation}
in a neighborhood of $\mathfrak{f}(\omega) = \omega^{-7/6}$, and with
a forcing satisfying
\begin{equation} \label{id:mainmassinput}
  \| \phi \|_{3/2-\delta,3/2+\delta}  \leq \epsilon .
\end{equation}
Expressed in terms of the perturbation \(g\) around \(\mathfrak{f}\),
i.e.\  $\mathfrak{f}$, ie $f = \mathfrak{f} + g$, the equation
\eqref{id:mainequation} becomes
\begin{equation}
  \label{eq:main-perturbation}
  \mathcal{L} g - \mathcal{Q}(g) - \mathcal{C}(g) = \phi,
\end{equation}
where we denote $\mathcal{L}$, $\mathcal{Q}$ and $\mathcal{C}$ for the
linear, quadratic and cubic terms. Using the power-law behaviour of
\(\mathfrak f\) the linear operator can be expressed as
\begin{equation}
  \label{eq:linearised-op}
  \mathcal{L} g(\omega)
  = \frac{a}{\omega^{1/3}} g
  - \frac{1}{\omega^{4/3}}
  \int_0^\infty k \left( \frac{r}{\omega} \right) f(r) \,\dd r
\end{equation}
where \(a>0\) and \(k\) can be expressed by explicit but lengthy
integrals as in
\cite[Appendix~A]{escobedo-mischler-velazquez-2007-uehling-uhlenbeck}. Using
the notation of \eqref{id:polarized}, the quadratic term can be
expressed as
\begin{equation*}
  \mathcal{Q}(g) = 3 \mathcal{C}(\mathfrak f,g,g).
\end{equation*}

This equation is more naturally expressed in the variables of
\eqref{eq:change-omega-ex} with
\begin{align*}
  \omega = e^x, \qquad g(k) = G(x), \qquad \phi(\omega) = \Phi(x), \qquad \mathcal{K}(x) = e^{-x} k(e^{-x})
\end{align*}
so that \eqref{eq:main-perturbation} becomes
\begin{equation}
  \label{eq:main-perturbation-x}
  {\mathcal{L}} G - e^{x/3}{\mathcal{Q}}(G) - e^{x/3}{\mathcal{C}}(G) = e^{x/3} \Phi,
\end{equation}
where
\begin{align*}
  {\mathcal{L}} &= a - \mathcal{K} *, \\
  {\mathcal{Q}}(G)(x) &= \mathcal{Q}(g)(\omega), \\
  {\mathcal{C}}(G)(x) &= \mathcal{C}(g)(\omega).
\end{align*}

\begin{nota}
  The reader will have noticed that we abuse notations by denoting
  $\mathcal{Q}$ and $\mathcal{C}$ for the quadratic and cubic
  operators, regardless of the coordinate choice being considered
  ($\omega$ or $x$). But this will be clear from the context, and
  upper case unknowns are always functions of $x$, while lower case
  indicates functions of $\omega$.
\end{nota}

Inverting ${\mathcal{L}}$, we arrive at the equivalent fixed-point
problem
\begin{equation}
  \label{fixedpoint}
  G = {\mathcal{L}}^{-1} e^{x/3} \left[  \Phi +  {\mathcal{Q}}(G) + {\mathcal{C}}(G)  \right].
\end{equation}

\begin{rem}
  The naive approach would be to apply the Banach fixed-point theorem
  in a space containing the constant term
  ${\mathcal{L}}^{-1} e^{x/3} \Phi$. Due to the constant term in
  \(\mathcal L\), one expects that \(\mathcal L^{-1}\) is a bounded
  operator from some \(E_{\alpha,\beta}\) to itself. For controlling
  the cubic term we thus would arrive to the following series of inequalities
  \begin{equation*}
    \| {\mathcal{L}}^{-1} e^{x/3} {\mathcal{C}}(G) \|_{\alpha,\beta}
    \lesssim \| e^{x/3} {\mathcal{C}}(G) \|_{\alpha,\beta}
    = \| \mathcal{C}(G) \|_{\alpha+\frac{1}{3},\beta+\frac 1 3} \lesssim \| G \|_{\alpha,\beta}^3.
  \end{equation*}
  where the last step only holds if \(\beta \ge 7/6\) by
  \cref{propC}. The same series of inequalities is needed for the
  quadratic term, where the last step only holds if
  \(\alpha \le 7/6\). But the function \(e^{\frac 76 x}\) is in the kernel of
  the linear operator \(\mathcal L\), which therefore cannot be inverted.

Of course, this obstacle should have been expected: it is related to the scaling invariance of the problem, or equivalently the need to modulate along the family of solutions $c \omega^{-7/6}$.
\end{rem}

The argument can be salvaged by improving it in
two respects: first, by peeling off the leading order behavior, we
find an inverse \(\mathcal L^{-1}\) from \(E_{\alpha,\beta}\) to
itself for \(\alpha=7/6-\delta_1\) and \(\beta=7/6+\delta_2\) and
second, by using the cancellations present in the equation, control
the leading order contribution.

We introduce $\chi$ a smooth cut-off function with $\chi(x)=1$ for
$x\geq 1$ and $\chi(x)=0$ for $x\leq 0$ and
\begin{equation*}
  G_0(x) = \chi(x) e^{-\frac{7}{6}x}.
\end{equation*}
In \cref{sec:linearised}, we will show that \(\mathcal L^{-1}\) can be
decomposed in the leading order.

\begin{prop}\label{lem:mainproofforcing}

  Let $0<\delta_1,\delta_2<1/6$. Then for any
  $\Phi \in E_{3/2-\delta_1,3/2+\delta_2}$, it is possible to
  decompose ${\mathcal{L}}^{-1} \left[ e^{x/3} \Phi \right]$ into
  leading order (as $x \to \infty$) plus remainder as
  \begin{equation*}
    {\mathcal{L}}^{-1} \left[ e^{x/3} \Phi \right] = c_1 M(\Phi) G_0(x) + \ell_1(\Phi),
  \end{equation*}
  where $c_1=-\frac{1}{3j_M^*}$, \(M(\Phi)\) is the mass integral
  \eqref{eq:mass-energy-integral} and \(\ell_1\) is a bounded linear
  operator from \(E_{3/2-\delta_1,3/2+\delta_2}\) to
  $E_{7/6-\delta_1, 7/6+\delta_2}$.
\end{prop}

The idea is now to extract the leading order term from the left and
right-hand sides of~\eqref{fixedpoint}: to this effect, we split the
sought perturbation $G$ into
\begin{equation*}
  G =  c G_0 + H.
\end{equation*}
Using \cref{lem:mainproofforcing}, we express the right-hand side of
\eqref{fixedpoint} for $F \in E_{\alpha,\beta}$ with
\(\alpha=3/2-\delta_1\) and \(\beta=3/2+\delta_2\) as
\begin{equation*}
  \mathcal{L}^{-1} e^{x/3} F = \ell_0(F) G_0 + \ell_1(F).
\end{equation*}
where \(\ell_0(F) = c_1 M(F)\). Introducing
\begin{equation} \label{id:defell}
  \ell =  \begin{pmatrix} \ell_0 \\ \ell_1 \end{pmatrix}
\end{equation}
we can now write the fixed point problem as
\begin{equation*}
  \begin{pmatrix} c \\ H \end{pmatrix}
  = \begin{pmatrix} \ell_0 \\ \ell_1 \end{pmatrix}
  ( \Phi +  {\mathcal{Q}}(G) + {\mathcal{C}}(G) )
  = \ell ( \Phi) + \ell (N(c,H)),
\end{equation*}
where
\begin{align*}
  N(c,H) & = {\mathcal{Q}}(G) + {\mathcal{C}}(G) \\
         & = c^2 \mathcal{Q}(G_0) +\mathcal{Q}(H)
           + 2 c \mathcal{Q}(G_0,H) + c^3 \mathcal C (G_0)
           +\mathcal{C}(H) + 3 c^2 \mathcal{C}(G_0,G_0,H)
           + 3 c \mathcal{C}(G_0,H,H).
\end{align*}

Adopting the natural norm on $\mathbb{R} \times E_{\alpha,\beta}$
\begin{equation*}
  \| (c,H) \|_{\alpha,\beta} = |c| + \| H \|_{\alpha,\beta},
\end{equation*}
the nonlinear term acts as a contraction:

\begin{prop}\label{pr:mainproofnonlinear}
 Assume
  $c\in \mathbb R$ and $H\in E_{\alpha,\beta}$ with
  $\alpha=7/6-\delta_1$ and $\beta=7/6+\delta_2$ where
  $0<\delta_1,\delta_2<1/12$. Then
  $\ell (N(c,H))\in E_{\alpha,\beta+\delta_2}$ with:
  \begin{equation} \label{bd:nonlinearterms}
    \left\| \ell (N(c,H)) \right\|_{\alpha+\delta_1,\beta} \lesssim \| (c,H) \|_{\alpha,\beta}^2 + \| (c,H) \|_{\alpha,\beta}^3.
  \end{equation}
\end{prop}

\begin{proof}
  We decompose
  $N(c,H) =N_1(c)+N_2(c,H)+\mathcal C(H) $ where
  \begin{align*}
    N_1(c) &= c^2 \mathcal{Q}(G_0) +c^3 \mathcal C (G_0) ,\\
    N_2(c,H) &= \mathcal{Q}(H) + 2 c \mathcal{Q}(G_0,H) + 3 c^2 \mathcal{C}(G_0,G_0,H) + 3 c \mathcal{C}(G_0,H,H).
  \end{align*}
  Using the nonlinear estimates from \cref{sec:nonlinear-estimates},
  we can estimate the different terms.

  By \cref{pr:G0estimates}, we conclude that $N_1(c)\in E_{7/6,5/3}$ with
  \begin{equation} \label{bd:nonlinearN1}
    \| N_1(c)\|_{7/6,5/3}\lesssim |c|^2+|c|^3.
  \end{equation}
  We now estimate all terms in $N_2$. We apply \cref{propQ} and obtain
  using $0<\delta_1,\delta_2<1/12$ that
  $\mathcal Q(H)\in E_{\alpha_1',\beta_1'}$,
  $\mathcal Q(G_0,H)\in E_{\alpha_2',\beta_2'}$,
  $\mathcal C(G_0,G_0,H)\in E_{\alpha_3',\beta_43}$ and
  $\mathcal C(G_0,H,H)\in E_{\alpha_4',\beta_4'}$ with:
  \begin{align*}
    & \alpha_1'=\frac 32-2\delta_1,
    & \beta_1'&=\frac 32+2\delta_2,
    & \| \mathcal Q(H)\|_{\alpha_1',\beta_1'} &\lesssim \| H\|_{\alpha,\beta}^2,\\
    & \alpha_2'=\frac 43-\delta_1,
    & \beta_2'&=\frac 32+\delta_2,
    & \| \mathcal Q(G_0,H)\|_{\alpha_2,\beta_2} &\lesssim \| H\|_{\alpha,\beta}, \\
    & \alpha_3'=\frac 76-\delta_1,
    & \beta_3'&=\frac 32 +\delta_2,
    & \| \mathcal C(G_0,G_0,H)\|_{\alpha_3',\beta_3'} &\lesssim  \| H\|_{\alpha,\beta}, \\
    & \alpha_4'=\frac 43-\delta_1,
    & \beta_4'&=\frac 32+2\delta_2,
    & \|  \mathcal C(G_0,H,H)\|_{\alpha_4',\beta_4'} &\lesssim \| H\|_{\alpha,\beta}^2.
  \end{align*}
  Combining, we obtain $N_2(c,H) \in E_{3/2-2\delta_1,3/2+\delta_2}$ with
  \begin{equation} \label{bd:nonlinearN2}
    \| N_2(c,H)\|_{E_{3/2-2\delta_1,3/2+\delta_2}}
    \lesssim \| H\|_{\alpha,\beta}\left(|c|+c^2+\| H\|_{\alpha,\beta}+|c|\| H\|_{\alpha,\beta}\right)
  \end{equation}
  We finally apply \cref{propC} using $0<\delta_1,\delta_2<1/12$ that
  $\mathcal C(H)\in E_{\alpha_5',\beta_5'}$ with
  \begin{equation} \label{bd:nonlinearCH}
    \alpha_5'=\frac 32-3\delta_1,
    \qquad \beta_5'=\frac 32+3\delta_2,
    \qquad  \|  \mathcal C(H)\|_{\alpha_5',\beta_5'}\lesssim \| H\|_{\alpha,\beta}^3.
  \end{equation}
  The three estimates \eqref{bd:nonlinearN1}, \eqref{bd:nonlinearN2}
  and \eqref{bd:nonlinearCH} imply
  $N(c,H)\in E_{3/2-2\delta_1,3/2+\delta_2}$ with
  \begin{equation*}
    \| N(c,H)\|_{3/2-2\delta_1,3/2+\delta_2}\lesssim \| (c,H) \|_{\alpha,\beta}^2 + \| (c,H) \|_{\alpha,\beta}^3.
  \end{equation*}
  We finally apply \cref{lem:mainproofforcing} and get
  $\ell (N(c,H))\in E_{7/6-2\delta_1,7/6+\delta_2}$ with the desired
  estimate~\eqref{bd:nonlinearterms}.
\end{proof}

We can now end the proof of our main theorem.

\begin{proof}[Proof of \cref{thm:main}]
  Consider the map
  $$
  \begin{array}{l l l l l}
    \Theta : &  \mathbb R\times E_{7/6-\delta,7/6+\delta} & \longrightarrow &  \mathbb R\times E_{7/6-\delta,7/6+\delta} \\
             & (c,H) & \mapsto & \ell (\Phi)+\ell (N(c,H)).
  \end{array}
  $$
  Combining \cref{lem:mainproofforcing}, \cref{pr:mainproofnonlinear}
  and the estimate \eqref{id:mainmassinput} shows that for all
  $0<\delta<1/12$, there exist $\epsilon^*>0$ and $K>0$ such that, for
  all $0<\epsilon\leq \epsilon^*$, the map $\Theta$ is a contraction
  on the ball of $\mathbb R\times E_{7/6-\delta,7/6+\delta}$ of radius
  $K\epsilon$. By the Banach fixed-point theorem, there exists a fixed
  point $(c^*,H^*)$. We define accordingly
  \begin{equation} \label{id:fixedpointsolution}
    f=\mathfrak f+g, \qquad g=c^*G_0+H^*,
  \end{equation}
  so that
  \begin{equation} \label{id:prooftheorem1}
    f(\omega)=
    \begin{cases}
      \omega^{-7/6}+O(\epsilon \omega^{-7/6+\delta})
      & \mbox{for } \omega\leq 1,\\
      (1+c^*) \omega^{-7/6}\left(1+O(\epsilon \omega^{-\delta})\right)
      &\mbox{for }\omega> 1.
    \end{cases}
  \end{equation}

  We have by definition of $\Theta $ that $f$ solves the equation on
  $(0,\infty)$:
  \begin{equation*}
    \forall \omega>0, \qquad \mathcal C(f)(\omega)=\phi(\omega).
  \end{equation*}
  The estimate \eqref{id:prooftheorem1} is the desired estimate
  \eqref{th:bd:stability} of the theorem. The remaining identities
  \eqref{th:id:massbalance}, \eqref{th:id:massflux},
  \eqref{th:id:energyflux} and \eqref{th:id:fluxes} are all
  consequences of \cref{pr:frompointwisetofluxes}. This ends the proof
  of \cref{thm:main} in the case $j_M^0=j_M^*$ i.e.
  $j_M^\infty=j_M^*-\int \sqrt{\omega}\phi\, \dd\omega$, and the general
  case $j_M^\infty\neq j_M^*-\int \sqrt{\omega}\phi\, \dd\omega$ follows
  by scaling invariance.
\end{proof}

\section{The linearized problem around $\omega^{-7/6}$} \label{sec:linearised}

In this section we prove \cref{lem:mainproofforcing}. The kernel $k$
of the linearised operator in \eqref{eq:linearised-op} is an analytic
function on \((0,1) \cup (1,\infty)\) with bounds
\begin{itemize}
\item \(|k(\lambda)| \lesssim \lambda^{1/3}\) for \(\lambda \ll 1\)
\item \(|k(\lambda)| \lesssim |\lambda-1|^{-5/6}\) for \(\lambda \to
  1\)
\item \(|k(\lambda)| \lesssim 1\) for \(\lambda \to \infty\)
\end{itemize}
This follows directly from the definition of \(k\) as written in
\cite[Appendix~A]{escobedo-mischler-velazquez-2007-uehling-uhlenbeck}. Note
that \cite{escobedo-mischler-velazquez-2007-uehling-uhlenbeck} obtain
further decay, but these more direct estimates are sufficient for us.

In terms of the new variable \(x\), this translates into the fact that
\begin{equation*}
  |\mathcal K(x)| \lesssim e^{-x} \quad \text{ as } x \to -\infty
  \qquad \text{and} \qquad
  |\mathcal K(x)| \lesssim e^{-4x/3} \quad \text{ as } x \to \infty.
\end{equation*}
Hence for any \(b \in (1,3/2)\) we find that
\begin{equation}
  \label{eq:integral-bound-k}
  \int_{x=-\infty}^{\infty} |\mathcal K(x)|\, e^{bx}\, \dd x < \infty.
\end{equation}

For the Fourier transform with the convention
\begin{equation*}
  \widehat{\mathcal K}(\xi) = \int_{-\infty}^\infty e^{-i x \xi} \mathcal K(x) \,\dd x
\end{equation*}
we therefore find that it is bounded and analytic in the strip
\begin{equation*}
  S_{\delta_1,\delta_2}
  = \{ z \in \mathbb C : \Im z \in [7/6-\delta_1,7/6+\delta_2] \}
\end{equation*}
for \(0<\delta_1,\delta_2<1/6\). Moreover, by a variant of the
Riemann-Lebesgue Lemma, see e.g.\ \cite[Theorem~2.8,
Chapter~2]{gripenberg-londen-staffans-1990-volterra},
\(\widehat{\mathcal K}(\xi) \to 0\) as \(|\Re \xi| \to \infty\) over
\(S_{\delta_1,\delta_2}\) .

To solve the linear problem
\begin{equation}
  \label{eq:linear-problem}
  \mathcal L G = \Psi
\end{equation}
we find formally in Fourier that
\begin{equation*}
  (a - \widehat{\mathcal{K}})(\xi)\, G(\xi) = \Psi(\xi)
\end{equation*}
yielding the spectral condition that
\(a - \widehat{\mathcal{K}} \not = 0\). Due to the analytic structure of
\(a-\widehat{\mathcal{K}}\), it is straightforward to numerically verify
the spectral condition by the argument principle on the strip
\(S_{\delta_1,\delta_2}\) and as shown in
\cite{escobedo-mischler-velazquez-2007-uehling-uhlenbeck} the only
zero in \(S_{\delta_1,\delta_2}\) for \(0<\delta_1,\delta_2<1/6\) is
at \(\xi=7i/6\), is a simple root, and corresponds to the stationary solution
\(\mathfrak f\).

For \(b \in (1,3/2)\) with \((a-\widehat K)(x+ib) \not = 0\) for all
\(x \in \R\), the bound \eqref{eq:integral-bound-k} implies that the
Paley-Wiener theorem
\cite[Chapter~2]{gripenberg-londen-staffans-1990-volterra} (as used
often for Volterra equations) yields a resolvent \(R_b\) with
\begin{equation*}
  \int_{-\infty}^{\infty} |R_{b}(x)|\, e^{bx}\, \dd x < \infty
\end{equation*}
such that, if $1 \leq p \leq \infty$ and \(\Psi \in L^p(\R,e^{bx}\, \dd x)\), then the problem
\eqref{eq:linear-problem} has a solution \(G \in L^p(\R,e^{bx}\, \dd
x)\) given by
\begin{equation*}
  G(x) = \frac{1}{a} (\Psi(x) + R_{b} * \Psi(x)).
\end{equation*}

Given \(0<\delta_1,\delta_2<1/6\), we can therefore find resolvents
\(R_-\) and \(R_+\) for the linear problem \eqref{eq:linear-problem}
with the bounds
\begin{equation*}
  \int_{-\infty}^{\infty} |R_{-}(x)|\, e^{(7/6-\delta_1)x}\, \dd x < \infty
  \quad\text{and}\quad
  \int_{-\infty}^{\infty} |R_{+}(x)|\, e^{(7/6+\delta_2)x}\, \dd x < \infty.
\end{equation*}
By the integrability and as \(R_{-}\) solves the linear problem, it
can be expressed through Fourier as
\begin{equation*}
  R_-(x) = \frac{1}{2\pi}
  \int_{-\infty}^{\infty} e^{ix(z+ib_{-})} \frac{\widehat
    K}{a-\widehat{K}}(z+ib_{-}) \,\dd z\qquad
  \text{for }
  b_{-} = 7/6-\delta_1.
\end{equation*}
Likewise for \(R_{+}\) we find
\begin{equation*}
  R_+(x) = \frac{1}{2\pi}
  \int_{-\infty}^{\infty} e^{ix(z+ib_{+})} \frac{\widehat
    K}{a-\widehat{K}}(z+ib_{+}) \,\dd z\qquad
  \text{for }
  b_{+} = 7/6+\delta_2.
\end{equation*}

As \(\widehat{K}\) is analytic and bounded in the strip
\(S_{\delta_1,\delta_2}\) with decay as \(|\Re z| \to \infty\) and
\(a-\widehat{K}\) has only a simple root in the strip, we find,
similar to \cite[Theorem~2.1,
Chapter~7]{gripenberg-londen-staffans-1990-volterra}, for a constant
\(c_1\) (from the residue) that
\begin{equation*}
  R_-(x) = R_+(x) + ac_1 e^{-7x/6}.
\end{equation*}
Hence we can write the solution to the linear problem \eqref{eq:linear-problem} as
\begin{equation*}
  \begin{split}
  G(x) &= \frac{\Psi}{a}(x) + \frac 1a \int_{y<0} R_-(y) \Psi(x-y)\, \dd y
  + \frac 1a \int_{y>0} R_-(y) \Psi(x-y)\, \dd y \\
  &= \frac{\Psi}{a}(x) + \frac 1a \int_{y<0} R_-(y) \Psi(x-y)\, \dd y
  + \frac 1a \int_{y>0} R_+(y) \Psi(x-y)\, \dd y
  + \int_{y>0} c_1 e^{-7y/6} \Psi(x-y)\, \dd y.
  \end{split}
\end{equation*}

The first three terms in the final expression satisfy the desired decay for
the remainder, i.e. each define a bounded linear operator from
\(E_{7/6-\delta_1,7/6+\delta_2}\) onto itself. The leading order comes
from the last integral which we express as
\begin{equation*}
  \begin{split}
    \int_{y>0} c_1 e^{-7y/6} \Psi(x-y)\, \dd y
    &= c_1 \chi(x) e^{-7x/6} \int_{y\in\mathbb R} \Psi(x-y) e^{7(x-y)/6}\,
    \dd y \\
    &\quad- c_1 \chi(x) e^{-7x/6} \int_{y<0} \Psi(x-y) e^{7(x-y)/6}\,
    \dd y \\
    &\quad+ (1-\chi(x)) c_1 e^{-7x/6} \int_{y>0} \Psi(x-y)
    e^{7(x-y)/6}\, \dd y.
  \end{split}
\end{equation*}
The first term is the sought leading order while the last two terms
can be included in \(\ell_1\) of \cref{lem:mainproofforcing}.

This shows the claimed splitting of \cref{lem:mainproofforcing}.

\begin{rem}[Value of \(c_1\)]
  We are going to use \cref{thm:main} and its proof, which are
  actually valid without knowing the exact value of $c_1$, to
  determine it. Indeed, pick any
  $\bar \Phi \in \mathcal C^\infty_0((0,\infty))$ with
  $M(\bar \Phi)=1$ and for $\epsilon>0$ let
  $\Phi_\epsilon=\epsilon \bar \Phi$ and
  \begin{equation} \label{bd:computationc12}
    j_{M,\epsilon}^\infty=1-\epsilon.
  \end{equation}
  For $\epsilon $ small, applying \cref{thm:main}, using the
  fixed-point relation \eqref{id:fixedpointsolution}, we get a
  solution of the form $f=\mathfrak f+c_\epsilon^*G_0+H_\epsilon^*$
  with
  \begin{equation} \label{bd:computationc11}
    |c_\epsilon^*|+\| H^*_\epsilon\|_{7/6-\delta^*,7/6+\delta^*}\lesssim \epsilon
  \end{equation}
  for some $\delta^*>0$ independent of $\epsilon$, and with $j_M^0=1$
  by mass balance \eqref{th:id:massbalance}. Still by the fixed-point
  relation \eqref{id:fixedpointsolution}, using \eqref{id:defell}, we
  obtain
  $c_\epsilon^*=\ell_0(\Phi_\epsilon+N(c_\epsilon^*,H_\epsilon^*))$. Using
  \cref{lem:mainproofforcing}, \cref{pr:mainproofnonlinear} and
  \eqref{bd:computationc11} this gives
  $c_\epsilon^*=c_1 \epsilon+O(\epsilon^2).$ Injecting this identity
  in the flux identity \eqref{id:weakformulationlimitfluxes} shows
  \begin{equation} \label{bd:computationc13}
    j^\infty_{M,\epsilon}=j_M^*(1+c_\epsilon^*)^3=j_M^*+3j_M^*c_1 \epsilon+O(\epsilon^2).
  \end{equation}
  The identities \eqref{bd:computationc12} and
  \eqref{bd:computationc13} impose that $c_1=-\frac{1}{3j_M^*}$ as
  desired.
\end{rem}

\section{Nonlinear estimates}
\label{sec:nonlinear-estimates}

\subsection{Estimates on the trilinear operator $\bar{\mathcal{C}}$}

Recall the trilinear operator \(\bar{\mathcal{C}}\) from
\eqref{eq:collision-3-4-trilinear}.  In the definition, we estimate
the absolute value of the integral as
\begin{equation}
  \label{eq:collision-3-4-trilinear-abs}
  |\bar{\mathcal{C}}(f,g,h)|
  \le
  \tilde{\mathcal{C}}(f,g,h) = 2 \iint_{0<\omega_3<\omega_4,0<\omega_2}
  W [(f_1+f_2) g_3 h_4 + (g_3+g_4) h_1 f_2]\, \dd \omega_3\, \dd \omega_4.
\end{equation}
The different estimates are then deduced from the following lemma.
\begin{lem}
  \label{thm:collision-trilinear-estimate}

  Suppose that \(f \in E_{\alpha_f,\beta_f}\),
  \(g \in E_{\alpha_g,\beta_g}\) and \(h \in E_{\alpha_h,\beta_h}\)
  for \(0\leq \alpha_f,\alpha_g,\alpha_h < 5/4\) with
  $\alpha_f,\alpha_g,\alpha_h \neq 1$ and
  \(\beta_f,\beta_g,\beta_h > 1\) with $\beta_f+\beta_g\neq 5/2$. Then
  for the absolute value of the collision integral
  \(\tilde{\mathcal{C}}\) we have
  \begin{equation*}
    \| \tilde{\mathcal{C}}(f,g,h) \|_{\alpha',\beta'}
    \lesssim
    \| f \|_{\alpha_f,\beta_f}
    \| g \|_{\alpha_g,\beta_g}
    \| h \|_{\alpha_h,\beta_g}
  \end{equation*}
  for
  \begin{equation*}
    \alpha' =
    \max\left(\alpha_f,\alpha_h,\alpha_f+\alpha_g+\alpha_h-2,\alpha_f+\alpha_g-1,\alpha_h+\alpha_g-1\right)
  \end{equation*}
  and
  \begin{equation*}
    \beta' =
    \min\left(\beta_f+\beta_g+\beta_h-2,\beta_h+\frac 12\right).
  \end{equation*}
\end{lem}

\begin{rem}
  The restrictions $\alpha_f,\alpha_g,\alpha_h \neq 1$ and
  $\beta_f+\beta_g\neq 5/2$ prevent logarithmic losses.
\end{rem}

\begin{proof}
It suffices to show that
  \(\| \tilde{\mathcal{C}}(f,g,h) \|_{\alpha',\beta'} \lesssim 1\)
  when \(f \le \rho_{\alpha_f,\beta_f}\), \(g \le \rho_{\alpha_g,\beta_g}\) and \(h \le \rho_{\alpha_h,\beta_h}\). We
  then estimate the different parts of the integral depending on
  whether \(\omega_1\), \(\omega_2\) or \(\omega_3\) is minimal.

  \medskip \noindent \underline{$\omega_1$ minimal:}
  In this case we find
  \begin{equation*}
    |\tilde{\mathcal{C}}^1|
    \lesssim
    \int_{\omega_3=\omega_1}^\infty \int_{\omega_4=\omega_3}^\infty
    \rho_{f,1} \rho_{g,3} \rho_{h,4}
    +
    \int_{\omega_3=\omega_1}^\infty \int_{\omega_2=2\omega_3-\omega_1}^\infty
    \rho_{h,1} \rho_{g,3} \rho_{f,2}.
  \end{equation*}
  For \(\omega_1 \ge 1\) we thus find
  \begin{equation*}
    |\tilde{\mathcal{C}}^1|
    \lesssim \omega_1^{2-\beta_f-\beta_g-\beta_h}
    \le \omega_1^{-\beta'}.
  \end{equation*}
  For \(\omega_1 \le 1\) we find
  \begin{align*}
    |\tilde{\mathcal{C}}^1|
    &\lesssim
      \omega_1^{-\alpha_f}
      \left(1 + \int_{\omega_3=\omega_1}^1
      \left(1 + \int_{\omega_4=\omega_3}^1 \rho_{h,4}\, \dd \omega_4
      \right)
      \rho_{g,3}\, \dd \omega_3 \right)\\
    &\qquad
      +
      \omega_1^{-\alpha_h}
      \left(1 + \int_{\omega_3=\omega_1}^1
      \left(1 + \int_{\omega_2=\omega_3}^1 \rho_{f,2}\, \dd \omega_4
      \right)
      \rho_{g,3}\, \dd \omega_3 \right)\\
    &\lesssim \omega_1^{-\alpha_f} +\omega_1^{-\alpha_h}+ \omega_1^{1-\alpha_f-\alpha_g}
      + \omega_1^{1-\alpha_g-\alpha_h}
      + \omega_1^{2-\alpha_f-\alpha_g-\alpha_h}
      \lesssim \omega_1^{-\alpha'}.
  \end{align*}

  \medskip \noindent \underline{$\omega_2$ minimal:}
  Here we find
  \begin{equation*}
    |\tilde{\mathcal{C}}^2|
    \lesssim
    \int_{\omega_2=0}^{\omega_1}
    \int_{\omega_3=\omega_2}^{(\omega_1+\omega_2)/2}
    \sqrt{\frac{\omega_2}{\omega_1}}\,
    [ \rho_{f,2} \rho_{g,3} \rho_{h,4}
    + \rho_{g,3} \rho_{h,1} \rho_{f,2} ].
  \end{equation*}
  In this case \(\omega_4 \ge \omega_1/2\) so that
  \begin{equation*}
    |\tilde{\mathcal{C}}^2|
    \lesssim
    \rho_{h,1}
    \int_{\omega_2=0}^{\omega_1}
    \int_{\omega_3=\omega_2}^{(\omega_1+\omega_2)/2}
    \sqrt{\frac{\omega_2}{\omega_1}}\,
    \rho_{g,3} \rho_{f,2}.
  \end{equation*}
  For \(\omega_1\le 1\) this shows
  \begin{equation*}
    |\tilde{\mathcal{C}}^2|
    \lesssim
    \omega_1^{2-\alpha_f-\alpha_g-\alpha_h}
    \le
    \omega_1^{-\alpha'}.
  \end{equation*}
  For \(\omega_1 \ge 1\) we find
  \begin{equation*}
    |\tilde{\mathcal{C}}^2|
    \lesssim
    \omega_1^{-\beta_h-\frac 12}
    \int_{\omega_2=0}^{\omega_1}
    \sqrt{\omega_2} \rho_{f,2}
    (\indic_{\omega_2 \le 1}(1+\omega_2^{1-\alpha_g}) + \indic_{\omega_2 \ge 1}
    \omega_2^{1-\beta_g})\,
    \dd \omega_2
    \lesssim \omega_1^{-\beta_h - \frac 12}
    + \omega_1^{2-\beta_f-\beta_g-\beta_h}
    \lesssim \omega_1^{-\beta'}.
  \end{equation*}

  \medskip \noindent \underline{$\omega_3$ minimal:}
  Here we find
  \begin{equation*}
    |\tilde{\mathcal{C}}^3|
    \lesssim
    \iint_{0<\omega_3<\omega_1,\omega_2,\omega_4}
    \sqrt{\frac{\omega_3}{\omega_1}}\,
    [ (\rho_{f,1}+\rho_{f,2}) \rho_{g,3} \rho_{h,4}
    + \rho_{g,3} \rho_{h,1} \rho_{f,2} ].
  \end{equation*}
  In this case \(\omega_4 \ge \omega_1\) so that
  \begin{equation*}
    |\tilde{\mathcal{C}}^3|
    \lesssim
    \int_{\omega_3=0}^{\omega_1} \int_{\omega_4=\omega_1}^\infty
    \sqrt{\frac{\omega_3}{\omega_1}}\,
    \rho_{f,1} \rho_{g,3} \rho_{h,4}
    +
    \int_{\omega_3=0}^{\omega_1} \int_{\omega_2=\omega_3}^\infty
    \sqrt{\frac{\omega_3}{\omega_1}}
    \rho_{h,1} \rho_{g,3} \rho_{f,2}.
  \end{equation*}
  Hence we find for \(\omega_1 \le 1\) that
  \begin{equation*}
    |\tilde{\mathcal{C}}^3|
    \lesssim
      \omega_1^{-\alpha_f-\frac 12}
      (1 + \omega_1^{1-\alpha_h})
      \int_{\omega_3=0}^{\omega_1} \omega_3^{\frac 12 - \alpha_g}\,
      \dd \omega_3
      +
      \omega_1^{-\alpha_h-\frac 12}
      \int_{\omega_3=0}^{\omega_1} \omega_3^{\frac 12 -\alpha_g}
      (1+\omega_3^{1-\alpha_f})
      \, \dd
      \omega_3
    \lesssim \omega_1^{-\alpha'}.
  \end{equation*}
  For \(\omega_1 \ge 1\) we find that
  \begin{equation*}
    |\tilde{\mathcal{C}}^3|
    \lesssim
      \omega_1^{-\beta_f - \frac 12}
      \omega_1^{1-\beta_h}
      (1 + \omega_1^{\frac 32 - \beta_g})
      + \omega_1^{-\beta_h-\frac 12}
      (1 + \omega_1^{\frac 52 -\beta_f - \beta_g})
    \lesssim \omega_1^{-\beta'}.
  \end{equation*}

  Hence we have found the required bound for all parts of the integral.
\end{proof}

\subsection{Nonlinear estimates on the collision operator}

We can deduce the immediate corollary for the collision operator.

\begin{cor} \label{propC}
Let \(0\le \alpha<5/4\) and \(\beta>1\) with $\alpha\neq 1$ and
  $\beta\neq 5/4$. Then the collision operator $\mathcal{C}$ is
  bounded from $E_{\alpha,\beta}$ to $E_{\alpha',\beta'}$, where
  \begin{equation*}
    \alpha' = \alpha + \max(0,2\alpha-2)
    \quad\text{and}\quad
    \beta' = \beta + \min(2\beta -2 ,\frac{1}{2}).
  \end{equation*}
  For \(\alpha<1\) there is no loss at small \(\omega\) as in this
  case \(\alpha' = \alpha\).
\end{cor}

On the remaining quadratic and cubic terms of the perturbation around
\(\mathfrak{f}\) we find the following direct corollary.
\begin{cor} \label{propQ}

  Let \(0\le \alpha<5/4\) and \(\beta>1\) with $\alpha\neq 1$ and
  $\beta \notin \{5/4,4/3\}$. Then one has the collection of estimates
  for $H\in E_{\alpha,\beta}$:
  \begin{align}
    \label{bd:nonlinear:QHH}
    \|\mathcal{Q}(H,H)\|_{\alpha_1',\beta_1'}
    &\lesssim \| H \|_{\alpha,\beta}^2,
    & \alpha_1' &= \max\left(2\alpha-\frac 56,\frac 76\right),
    & \beta_1' &= \min\left(\frac 53,\beta+\frac 12,2\beta-\frac{5}{6}\right),\\
    \label{bd:nonlinear:QG0H}
    \|\mathcal{Q}(G_0,H)\|_{\alpha_2',\beta_2'}
    &\lesssim \| H \|_{\alpha,\beta},
    & \alpha_2' &= \max\left(\alpha+\frac 1 6,\frac 76\right),
    & \beta_2' &= \min\left(\frac 53,\beta+\frac 13\right),\\
    \label{bd:nonlinear:CG0G0H}
    \|\mathcal{C}(G_0,G_0,H)\|_{\alpha_3',\beta_3'}
    &\lesssim \| H \|_{\alpha,\beta},
    & \alpha_3'&=\alpha,
    & \beta_3'&=\min\left(\frac 53,\beta+\frac 13\right),\\
    \label{bd:nonlinear:CG0HH}
    \|\mathcal{C}(G_0,H,H)\|_{\alpha_4',\beta_4'}
    &\lesssim \| H \|_{\alpha,\beta}^2,
    & \alpha_4'&=\max\left(\alpha,2\alpha-1\right),
    & \beta_4'&=\min\left(\frac 53,\beta+\frac 12,2\beta-\frac 56\right).
  \end{align}
\end{cor}

\begin{proof}
  Recall the identity \eqref{id:polarized}. The first inequality
  \eqref{bd:nonlinear:QHH} is obtained by applying Lemma
  \ref{thm:collision-trilinear-estimate} to
  $(f,g,h)\in \{ (\mathfrak f,H,H),(H,\mathfrak f,H),(H,H,\mathfrak
  f)\}$, and with $(\alpha_f,\alpha_h,\beta_h)$ and
  $(\beta_f,\beta_g,\beta_h)$ the corresponding permutations of
  $(7/6,\alpha,\alpha)$ and $(7/6,\beta,\beta)$. The remaining
  estimates are obtained similarly, using in addition that
  $G_0\in E_{0,7/6}$.
\end{proof}

\subsection{Nonlinear estimates involving $g_0$}

\begin{prop}\label{pr:G0estimates}
  The function $\mathcal{C}(g_0)$ belongs to $E_{0,5/3}$ and the
  function $\mathcal{C}(\mathfrak{f},g_0,g_0)$ belongs to
  $E_{7/6,5/3}$.
\end{prop}

\begin{proof}

  We first deal with $\mathcal{C}(g_0)$. Notice that
  $g_0 \in E_{0,7/6}$, so that by applying \cref{propC} we
  get
  \begin{equation} \label{id:mnonlinearg03}
    \mathcal{C}(g_0)\in E_{0,3/2}.
  \end{equation}
  This shows that $\mathcal C(g_0)$ stays uniformly bounded for
  $\omega$ small. To deal with large frequencies $\omega \gg 1$, we
  let
  $$
  g_1 = \mathfrak{f} - g_0 \qquad \mbox{so that } \supp g_1 = [0,c], \mbox{ for some } c>0
  $$
  and we rely on the identity $\mathcal{C}(\mathfrak{f}) = 0$, whose
  proof is recalled in \cref{basiccancellation} to write
  \begin{equation} \label{id:mnonlinearg01}
    \mathcal{C}(g_0) = - \mathcal{C}(g_1) - 3 \mathcal{C}(g_0,g_1,g_1) - 3 \mathcal{C}(g_0,g_0,g_1) .
  \end{equation}
  Above, we have $g_1\in E_{7/6,M}$ where $M>0$ is any arbitrarily
  large constant, so that applying \cref{propC} to $g_1$, and then the
  estimates \eqref{bd:nonlinear:CG0G0H} and \eqref{bd:nonlinear:CG0HH}
  we obtain
  \begin{equation}\label{id:mnonlinearg02}
    \mathcal{C}(g_1)  \in E_{3/2,M}, \qquad  \mathcal{C}(g_0,g_1,g_1)\in E_{4/3,5/3}, \qquad  \mathcal{C}(g_0,g_0,g_1)\in E_{7/6,5/3}.
  \end{equation}
  Injecting \eqref{id:mnonlinearg02} in \eqref{id:mnonlinearg01} shows
  $\mathcal{C}(g_0)\in E_{3/2,5/3}$. Combining with
  \eqref{id:mnonlinearg03} we obtain
  $\mathcal{C}(g_0)\in E_{0,3/2}\cap E_{3/2,5/3}=E_{0,5/3}$. This is
  the first estimate of the proposition.

  We now deal with $\mathcal{C}(\mathfrak f, g_0,g_0)$ in a similar
  way. As $g_0\in E_{0,7/6}$, applying estimate
  \eqref{bd:nonlinear:QHH} gives
  \begin{equation} \label{id:mnonlinearg04}
    \mathcal{C}(\mathfrak f, g_0,g_0)\in E_{7/6,3/2}.
  \end{equation}
  Again, this is sufficient for $\omega$ small. For $\omega$ large we
  once more use the identity $\mathcal{C}(\mathfrak{f}) = 0$ to
  produce
  \begin{equation} \label{id:mnonlinearg05}
    \mathcal{C}(\mathfrak{f},g_0,g_0) = - 2 \mathcal{C}(\mathfrak{f},g_0,g_1) - \mathcal{C}(\mathfrak{f},g_1,g_1).
  \end{equation}
  We apply \eqref{bd:nonlinear:QG0H} and \eqref{bd:nonlinear:QHH} and get
  \begin{equation} \label{id:mnonlinearg06}
    \mathcal{C}(\mathfrak{f},g_0,g_1) \in E_{4/3,5/3}, \qquad
    \mathcal{C}(\mathfrak{f},g_1,g_1)\in E_{3/2,5/3}.
  \end{equation}
  Injecting \eqref{id:mnonlinearg06} in \eqref{id:mnonlinearg05} shows
  $ \mathcal{C}(\mathfrak{f},g_0,g_0) \in E_{3/2,5/3}$. Combined with
  \eqref{id:mnonlinearg04} this shows
  $\mathcal{C}(\mathfrak f, g_0,g_0)\in E_{7/6,5/3}$. This is the
  second estimate and ends the proof.
\end{proof}

\subsection{The basic cancellation}
\label{basiccancellation}

The KZ spectrum $\mathfrak f=\omega^{-7/6}$ is a stationary solution
of \eqref{KWE}. We give here a proof relying solely on the scaling
invariance of the equation and on the locality of the integrals of
collision and mass flux (in the spirit of the dimensional analysis
argument of \cite{ZLF} and
\cite[Prop.~2.41]{escobedo-velazquez-2015-schroedinger}), before
recalling the classical one.

\begin{lem}\label{lem:KZisstationary}
  For all $\omega>0$ one has $\mathcal C(\mathfrak f)(\omega)=0$.
\end{lem}

\begin{proof}
  By \cref{proplocal} and \cref{lem:localflux}, we
  have for all $\omega>0$ that $\mathcal C(\mathfrak f)(\omega)$ as
  well as $J_M(\mathfrak f)(\omega)$ are well-defined, i.e. by
  convergent integrals. By rescaling the variables in the integrals in
  \eqref{def:JM}, we get that
  $J(\mathfrak f)(\omega)=J(\mathfrak f)(1)$ is actually independent
  of $\omega>0$. Hence by differentiating,
  $\partial_\omega J(\mathfrak f)=0$. Thus \eqref{id:massfluxduality}
  shows that \(\sqrt{\omega}\mathcal C(\mathfrak
  f)(\omega)=0\) almost everywhere.
\end{proof}

We recall now the classical conformal change of coordinates due to
Zakharov~\cite{zakharov-1972-langmuir}, which shows that
$\omega^{-7/6}$ and $\omega^{-3/2}$ are solutions of
$\mathcal{C}(f) = 0$. In the former case, the spectrum is local
(\cref{proplocal}), and the computation is rigorously
justified; in the latter case, it is merely formal.

We assume for the solution the form $f(\omega) = \omega^\alpha$ so
that the collision operator \eqref{collisionoperator} reads
\begin{equation*}
  \mathcal{C}(f)(\omega_1)
  = \iint_{\omega_2,\omega_3,\omega_4 \geq 0}
  W \omega_1^{\alpha} \omega_2^\alpha \omega_3^\alpha \omega_4^\alpha
  \left[ \omega_1^{-\alpha} + \omega_2^{-\alpha} - \omega_3^{-\alpha}
    - \omega_4^{-\alpha} \right]
  \, \dd\omega_3 \,\dd\omega_4.
\end{equation*}
Then we split the domain of the integral defining the collision
operator~\eqref{collisionoperator} into four subdomains as
\begin{equation*}
  \mathcal{C}(f)
  = \int_{\substack{\omega_2,\omega_3,\omega_4 \geq 0 \\
      \omega_1+\omega_2 = \omega_3 + \omega_4}}
  \dots \dd\omega_3 \, \dd\omega_4
  = \sum_{i=1}^4 \int_{\Delta_i} \dots \dd\omega_3 \, \dd\omega_4,
\end{equation*}
where
\begin{itemize}
\item $\Delta_1 = \{ (\omega_3,\omega_4) \; \mbox{such that} \; \omega_3 + \omega_4 \geq \omega_1, \; 0 \leq \omega_3 \leq \omega_1,  \; 0 \leq \omega_4 \leq \omega_1 \}$
\item $\Delta_2 = \{ (\omega_3,\omega_4) \; \mbox{such that} \; \omega_3 \geq \omega_1,  \; 0 \leq \omega_4 \leq \omega_1 \}$
\item $\Delta_3 = \{ (\omega_3,\omega_4) \; \mbox{such that} \; \omega_3 \geq \omega_1, \; \omega_4 \geq \omega_1 \}$
\item $\Delta_4 = \{ (\omega_3,\omega_4) \; \mbox{such that}  \; 0 \leq \omega_3 \leq \omega_1, \;  \omega_4 \geq \omega_1 \}$
\end{itemize}
which is illustrated in \cref{fig:conformal}.

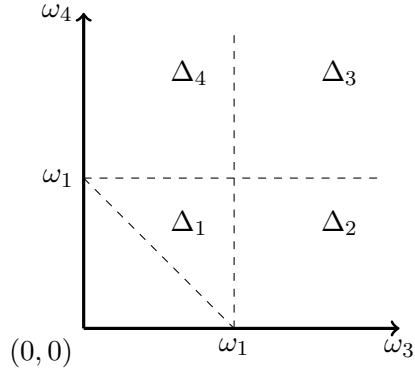
\begin{figure}[htb]
  \centering
  \begin{tikzpicture}[scale=2]
    \draw[very thick,->] (0,0) -- (2.1,0) node[anchor=north] {\(\omega_3\)};
    \draw[very thick,->] (0,0) -- (0,2.1) node[anchor=east]
    {\(\omega_4\)};
    \draw[dashed] (0,1) -- (1,0);
    \draw[dashed] (0,1) -- (2,1);
    \draw[dashed] (1,0) -- (1,2);
    \draw (0.7,0.7) node {\(\Delta_1\)};
    \draw (1.7,0.7) node {\(\Delta_2\)};
    \draw (1.7,1.7) node {\(\Delta_3\)};
    \draw (0.7,1.7) node {\(\Delta_4\)};
    \draw (0,0) node[anchor=north east] {\((0,0)\)};
    \draw (1,0) node[anchor=north] {\(\omega_1\)};
    \draw (0,1) node[anchor=east] {\(\omega_1\)};
  \end{tikzpicture}
  \caption{Illustration of the regions \(\Delta_1\), \(\Delta_2\),
    \(\Delta_3\), \(\Delta_4\) in the collision integral for
    \((\omega_3,\omega_4)\). Through an appropriate change of
    variables they will all mapped onto \(\Delta_1\) where they can
    cancel.}
  \label{fig:conformal}
\end{figure}

The region \(\Delta_2\) can be mapped to \(\Delta_1\) by the the change of variables
\begin{equation*}
  (\omega_1,\omega_2',\omega_3',\omega_4')
  = \frac{\omega_1}{\omega_3} (\omega_3,\omega_4,\omega_1,\omega_2).
\end{equation*}
The Jacobian is
\begin{equation*}
  \det \frac{\partial(\omega_3',\omega_4')}{\partial(\omega_3,\omega_4)}
  = - \frac{\omega_3^{'3}}{\omega_1^3}
\end{equation*}
and the new coordinates satisfy the conservation of momentum as
$\omega_1+ \omega_2' = \omega_3' + \omega_4'$.  By applying this
change of variables and omitting the primes, we thus find
\begin{equation*}
  \begin{aligned}
    \int_{\Delta_2}
    &W \omega_1^{\alpha} \omega_2^\alpha \omega_3^\alpha \omega_4^\alpha
    \left[ \omega_1^{-\alpha} + \omega_2^{-\alpha} - \omega_3^{-\alpha}
    - \omega_4^{-\alpha} \right]
    \,\dd\omega_3 \, \dd\omega_4 \\
    &= - \int_{\Delta_1}
      W \omega_1^{\alpha} \omega_2^\alpha \omega_3^\alpha
      \omega_4^\alpha
      \left[ \omega_1^{-\alpha} + \omega_2^{-\alpha} -
      \omega_3^{-\alpha} - \omega_4^{-\alpha} \right]
      \left( \frac{\omega_1}{\omega_3} \right)^{3\alpha + \frac{7}{2}}
      \,\dd\omega_3 \, \dd\omega_4.
  \end{aligned}
\end{equation*}

The region \(\Delta_3\) can be mapped to \(\Delta_1\) by the the
change of variables
\begin{equation*}
  (\omega_1,\omega_2',\omega_3',\omega_4')
  = \frac{\omega_1}{\omega_2} (\omega_2,\omega_1,\omega_3,\omega_4).
\end{equation*}
The Jacobian is
\begin{equation*}
  \det \frac{\partial(\omega_3',\omega_4')}{\partial(\omega_3,\omega_4)}
  = - \frac{\omega_2^{'3}}{\omega_1^3}
\end{equation*}
and the new coordinates satisfy the conservation of momentum as
$\omega_1+ \omega_2' = \omega_3' + \omega_4'$.  By applying this
change of variables and omitting the primes, we thus find
\begin{equation*}
  \begin{aligned}
    \int_{\Delta_3}
    &W \omega_1^{\alpha} \omega_2^\alpha \omega_3^\alpha \omega_4^\alpha
    \left[ \omega_1^{-\alpha} + \omega_2^{-\alpha} - \omega_3^{-\alpha}
    - \omega_4^{-\alpha} \right]
    \,\dd\omega_3 \, \dd\omega_4 \\
    &= \int_{\Delta_1}
      W \omega_1^{\alpha} \omega_2^\alpha \omega_3^\alpha
      \omega_4^\alpha
      \left[ \omega_1^{-\alpha} + \omega_2^{-\alpha} -
      \omega_3^{-\alpha} - \omega_4^{-\alpha} \right]
      \left( \frac{\omega_1}{\omega_2} \right)^{3\alpha + \frac{7}{2}}
      \,\dd\omega_3 \, \dd\omega_4.
  \end{aligned}
\end{equation*}

The integral over \(\Delta_4\) is symmetric to \(\Delta_2\). Hence
collecting the mappings we obtain that
\begin{align*}
  \mathcal{C}(f)
  & = \int_{\Delta_1}  W \omega_1^{\alpha} \omega_2^\alpha \omega_3^\alpha \omega_4^\alpha \left[ \omega_1^{-\alpha} + \omega_2^{-\alpha} - \omega_3^{-\alpha} - \omega_4^{-\alpha} \right] \\
  & \qquad \qquad\qquad \qquad \left[1+ \left( \frac{\omega_1}{\omega_2} \right)^{3\alpha + \frac{7}{2}} - \left( \frac{\omega_1}{\omega_3} \right)^{3\alpha + \frac{7}{2}} - \left( \frac{\omega_1}{\omega_4} \right)^{3\alpha + \frac{7}{2}}  \right] \,\dd\omega_3 \, \dd\omega_4.
\end{align*}
From this expression, one can immediately read off the values of
$\alpha$ for which $k^{-\alpha}$ is a (formal) stationary solution:
\begin{itemize}
\item Either $-\alpha = 0$ or $1$, which gives the RJ solutions
  $\alpha = 0$ and $\alpha = - 1$.
\item Or $-3\alpha - \frac{7}{2} = 0$ or
  $-3\alpha - \frac{7}{2} = -1$, which gives the KZ solution
  $\alpha = - \frac{7}{6}$ and $\alpha = - \frac{3}{2}$.
\end{itemize}

\section*{Acknowledgements}

CC is supported by the CY Initiative of Excellence Grant
``Investissements d'Avenir'' ANR-16-IDEX-0008. Part of this work was
done while CC was visiting the Courant Institute for Mathematical
Sciences that he would like to thank, as well as the Simons
Collaboration on Wave Turbulence for its financial support.

HD would like to thank the Isaac Newton Institute for Mathematical
Sciences, Cambridge, for support and hospitality during the programme.
This work was supported by EPSRC grant no EP/R014604/1 and a grant
from the Simons Foundations.  HD acknowledge the grant
ANR-18-CE40-0027 of the French National Research Agency (ANR).

PG is supported by the Simons collaborative grant on wave turbulence.

\bibliographystyle{abbrv}
\bibliography{lit}

\end{document}